\documentclass{amsart}
\usepackage{array}
\newcolumntype{P}[1]{>{\raggedright\let\newline\\\arraybackslash\hspace{0pt}}m{#1}}
\usepackage{graphicx,verbatim, amsmath, amssymb, amsthm, amsfonts, epsfig, amsxtra,ifthen,mathtools,
caption,enumerate,
hhline,bbm,capt-of,longtable,enumerate}	
\DeclareFontFamily{U}{mathx}{\hyphenchar\font45}
\DeclareFontShape{U}{mathx}{m}{n}{<-> mathx10}{}
\DeclareSymbolFont{mathx}{U}{mathx}{m}{n}
\DeclareMathAccent{\widebar}{0}{mathx}{"73}
\usepackage[bookmarks=true, bookmarksopen=false,%
    colorlinks=true,%
    linkcolor=darkblue,%
    citecolor=darkblue,%
    filecolor=darkblue,%
    menucolor=darkblue,%
    linktoc=all,%
    urlcolor=darkblue
]{hyperref}

\usepackage[usenames]{color}
\definecolor{darkblue}{cmyk}{1,0.3,0,0.1}  

\newtheorem{proposition}{Proposition}[section]
\newtheorem{theorem}[proposition]{Theorem}
\newtheorem{corollary}[proposition]{Corollary}
\newtheorem{lemma}[proposition]{Lemma}
\newtheorem{prop}[proposition]{Proposition}

\newtheorem{thm}[proposition]{Theorem}

\theoremstyle{definition}
\newtheorem{example}[proposition]{Example}
\newtheorem{definition}[proposition]{Definition}
\newtheorem{defcon}[proposition]{Definition/Construction}

\theoremstyle{remark}
\newtheorem{remark}[proposition]{Remark}

\numberwithin{equation}{section}
\usepackage[usenames]{color}

\usepackage{color}

\newcommand{\margincolor}{red}      
\definecolor{darkgreen}{rgb}{0,0.7,0}

\newcounter{margincounter}
\newcommand{\marginnum}{
\ifnum\value{margincounter}<10
\textcolor{\margincolor}{\begin{picture}(0,0)\put(2.2,2.4){\circle{9}}\end{picture}\footnotesize\arabic{margincounter}}
\else\ifnum\value{margincounter}<100
\textcolor{\margincolor}{\begin{picture}(0,0)\put(4.256,2.5){\circle{11}}\end{picture}\footnotesize\arabic{margincounter}}
\else
\textcolor{\margincolor}{\begin{picture}(0,0)\put(6.8,2.5){\circle{14}}\end{picture}\footnotesize\arabic{margincounter}}
\fi\fi
}

\newcommand{\newword}[1]{\textbf{\emph{#1}}}

\newcommand{\rank}{\operatorname{rank}}

\newcommand{\covered}{{\,\,<\!\!\!\!\cdot\,\,\,}}
\newcommand{\set}[1]{{\lbrace #1 \rbrace}}

\newcommand{\A}{{\mathcal A}}

\newcommand{\D}{{\mathbf D}}

\newcommand{\join}{\vee}
\newcommand{\meet}{\wedge}
\renewcommand{\Join}{\bigvee}

\newcommand{\B}{\mathbf{B}}
\newcommand{\M}{\mathbf{M}}
\newcommand{\R}{\mathcal{R}}

\renewcommand{\P}{\mathcal{P}}

\newcommand{\Q}{\mathcal{Q}}

\renewcommand{\S}{\mathbf{S}}

\newcommand{\NC}{\operatorname{NC}}
\newcommand{\curve}{\operatorname{curve}}

\renewcommand{\int}{{\operatorname{int}}}
\renewcommand{\phi}{\varphi}
\newcommand{\afftype}[1]{{\widetilde{\raisebox{0pt}[6pt][0pt]{#1}}}}

\title{Noncrossing partitions of a marked surface}
\author{Nathan Reading}
\thanks{Nathan Reading was partially supported by the National Science Foundation under Award Number DMS-2054489.}

\begin{document}

\begin{abstract}
We define noncrossing partitions of a marked surface without punctures (interior marked points).
We show that the natural partial order on noncrossing partitions is a graded lattice and describe its rank function topologically.
Lower intervals in the lattice are isomorphic to products of noncrossing partition lattices of other surfaces.
We similarly define noncrossing partitions of a symmetric marked surface with double points and prove some of the analogous results.
The combination of symmetry and double points plays a role that one might have expected to be played by punctures.
\end{abstract}
\maketitle

\vspace{-15pt}

\setcounter{tocdepth}{1}
\tableofcontents

\section{Introduction}\label{intro sec}  
In this paper, we introduce two closely related generalizations of Kreweras' classical construction of noncrossing partitions of a cycle~\cite{Kreweras}.
The first generalization replaces the cycle (or, more correctly, a disk) with a general marked surface (without punctures) in the sense of \cite{cats1,cats2}.
A marked surface is an oriented compact surface $\S$ with boundary and a finite set $\M$ of distinguished boundary points.  
A noncrossing partition of a marked surface $(\S,\M)$ is a set partition of the set of marked points together with an embedding of each block as a point, curve, or surface in $(\S,\M)$.
Except in the simplest cases, there are infinitely many noncrossing partitions.
For precise definitions, see Section~\ref{marked sec}.
For now, we show in Figure~\ref{nc ex fig} some noncrossing partitions in the case where $\S$ is a sphere minus three open disks.  
Here, as in other figures, blocks that are embedded as points or curves are given more thickness, to improve legibility.

\begin{figure}
\scalebox{0.7}{\includegraphics{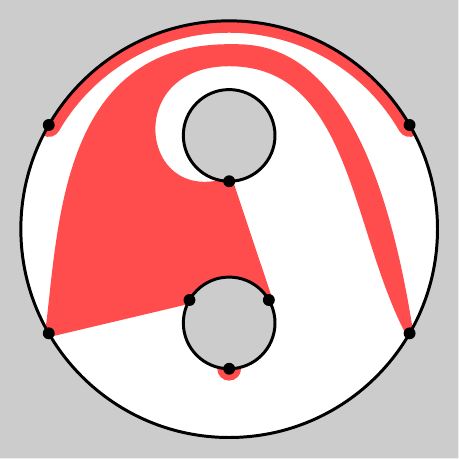}}
\quad\quad
\scalebox{0.7}{\includegraphics{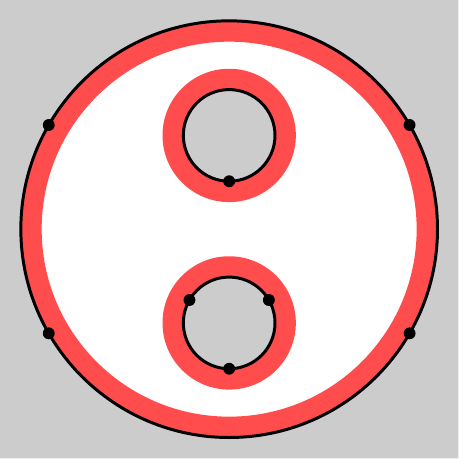}}\\[20pt]
\scalebox{0.7}{\includegraphics{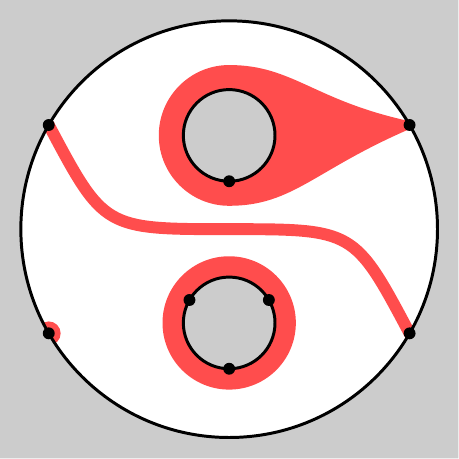}}
\quad\quad
\scalebox{0.7}{\includegraphics{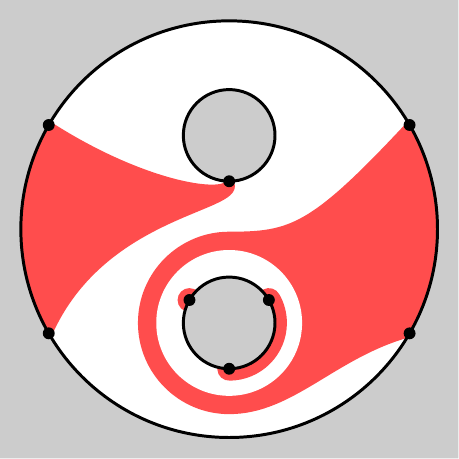}}
\caption{Some noncrossing partitions of a marked surface}
\label{nc ex fig}
\end{figure}

There is a natural partial order on noncrossing partitions (containment of the embedded blocks).
We show that this partial order is a lattice (Theorem~\ref{lattice}) and that it is graded (Theorem~\ref{graded}).
The rank function is the number of marked points minus the first Betti number plus the zeroth Betti number.
We observe (Proposition~\ref{lower}) that lower intervals in the noncrossing partition lattice are isomorphic to products of noncrossing partition lattices for marked surfaces.

The second generalization, defined more precisely in Section~\ref{doub sec}, instead features a surface with a nontrivial involutive symmetry and allows, in addition to marked points on the boundary, special points in the interior called double points, which occur in pairs at the same location in the surface.
Crucially, the symmetry is required to map each double point to the other double point in the same location.
In this setting, a noncrossing partition is a set partition of the marked points (including the boundary marked points and both copies of each double point) enriched with an embedding of each block as a point, curve, marked surface, or symmetric marked surface with double points.
Some examples of symmetric noncrossing partitions of a marked surface with double points are shown in Figure~\ref{sym nc ex fig}.
The figure shows a sphere with four open disks removed.
The involutive symmetry is a two-fold rotation that swaps the outer and inner circles shown and also swaps the smaller side circles.
This symmetry has two fixed points.  
One, shown lower in the pictures, is a double point, and the other is marked with a small dotted circle higher in the pictures.

\begin{figure}
\scalebox{0.7}{\includegraphics{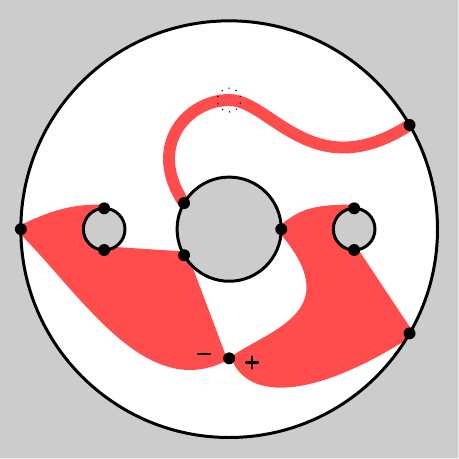}}
\quad\quad
\scalebox{0.7}{\includegraphics{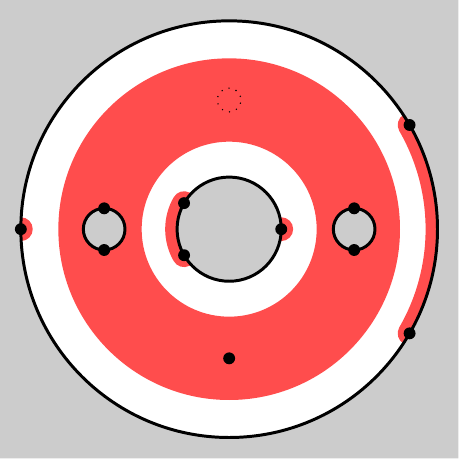}}\\[20pt]
\scalebox{0.7}{\includegraphics{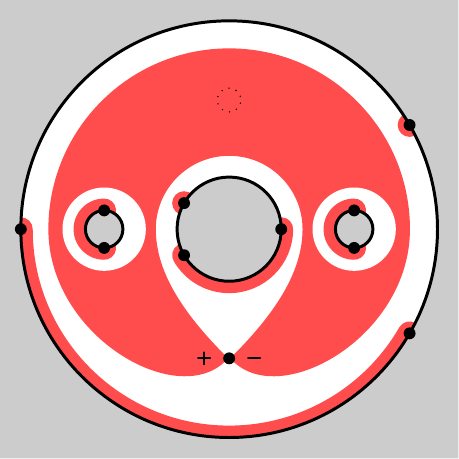}}
\quad\quad
\scalebox{0.7}{\includegraphics{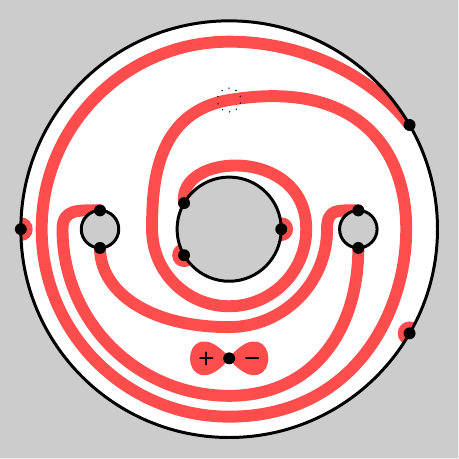}}
\caption{Some symmetric noncrossing partitions of a marked surface with double points}
\label{sym nc ex fig}
\end{figure}

In the symmetric-with-double-points setting, the natural partial order on noncrossing partitions is not always a lattice.
(See Example~\ref{non lat}.)
However, it is still graded (Theorem~\ref{sym graded}).
The rank function in this case also involves something like Betti numbers, specifically the dimensions of the kernel of a certain linear map on homology.
Again, lower intervals in the noncrossing partition poset are isomorphic to products of noncrossing partition posets (Proposition~\ref{sym lower}).

This paper is a prequel to~\cite{BThesis}, in the sense that it was written later but (from one point of view) it logically precedes~\cite{BThesis}.
The general constructions and proofs given here unify constructions and proofs due to Laura Brestensky and the author on planar models for noncrossing partitions of classical affine types (in joint work and in Brestensky's Ph.D. thesis research~\cite{BThesis}).
In particular, the results of this paper will be used to present results on noncrossing partitions of classical affine types more concisely in~\cite{affncA,affncD}.

Earlier sources of inspiration for noncrossing partitions of marked surfaces (with and without symmetry/double points), in addition to Kreweras~\cite{Kreweras}, include the finite type-B and type-D noncrossing partitions of Reiner~\cite{Reiner} and Athanasiadis-Reiner~\cite{Ath-Rei}, and the Digne's noncrossing paths in the annulus~\cite{Digne1}.
Our work on noncrossing partitions of affine type began as a project to create combinatorial models for the work of McCammond and Sulway~\cite{McSul}.

The guiding principle in defining noncrossing partitions of a marked surface with and without symmetry/double points was to reproduce the behavior of noncrossing partitions associated to finite and affine Coxeter groups.

\subsection*{Organization of this paper}
In Section~\ref{marked sec}, we define noncrossing partitions of a marked surface and establish their properties as summarized above.
Then, in Section~\ref{doub sec}, we define noncrossing partitions of a symmetric marked surface with double points and establish their properties, with definitions and arguments parallel to those in Section~\ref{marked sec} and reusing arguments as much as possible.
In Section~\ref{ex sec}, we revisit the motivating examples (noncrossing partitions of finite types A, B, and D and affine types $\afftype{A}$, $\afftype{B}$, $\afftype{C}$, and $\afftype{D}$) in light of the general results.
In the affine cases, the relationship between these planar models and affine noncrossing partitions is given in \cite{BThesis} and will appear in \cite{affncA,affncD}.
Also in that section (Remark~\ref{clus rem}), we briefly discuss the similarities and differences between noncrossing partitions and the combinatorics of clusters (triangulations) in these motivating examples, including comparing double points/symmetry with punctures.

\section{Noncrossing partitions of a marked surface}\label{marked sec}  
In this section, we define noncrossing partitions of a marked surface and establish their properties as summarized in Section~\ref{intro sec}.
Although the definition of noncrossing partitions of a surface is new, it uses some background material from \cite{cats1,cats2}, which considers (tagged) triangulations of marked surfaces.
We allow, as part of our definition of marked surfaces, some degenerate cases (points and curves).

\begin{definition}[\emph{Nondegenerate marked surface}]\label{nondegen def}
A \newword{nondegenerate marked surface} is $(\S,\M)$, where $\S$ is a surface and $\M$ is a set of points of $\S$ called \newword{marked points}, satisfying conditions that we now describe.
The surface $\S$ is compact, connected, and oriented, and has a boundary $\partial\S$ that is a union of disjoint circles called \newword{boundary components}.
The set $\M$ of marked points is nonempty and finite and contained in $\partial\S$.
(A marked point in the interior of $\S$ is called a \newword{puncture}; the marked surfaces considered in this paper have no punctures, but later, in Section~\ref{doub sec}, we consider double points in the interior of~$\S$.)
We rule out the cases where $\S$ is a disk with $1$, or $2$ marked points.
In many contexts (e.g.\ in the foundational papers \cite{cats1,cats2} on cluster algebras and marked surfaces), there is a requirement that every boundary component contains at least one marked point.
We don't make that requirement, and we call a boundary component of $\S$ with no marked points an \newword{empty boundary component}.
A piece of the boundary of~$\S$ with endpoints in $\M$ but no points of $\M$ in its interior is called a \newword{boundary segment}.
A boundary segment may have both endpoints at the same marked point.
Since empty boundary components contain no marked points, they also contain no boundary segments.
\end{definition}

\begin{definition}[\emph{Degenerate marked surface}]\label{degen marked surf def}
A \newword{degenerate marked surface} is $(\S,\M)$, where either
\begin{itemize}
\item
$\S$ is a one-element set and $\M=\S$, or
\item
$\S$ is a curve with two distinct endpoints, and $\M$ is the set consisting of the two endpoints of $\S$.
\end{itemize}
In either case, the elements of $\M$ are called \newword{marked points}.
When $\S$ is a singleton, there are no boundary segments. 
When $\S$ is a curve, it has a unique boundary segment, which is $\S$ itself. 
\end{definition}

\begin{definition}[\emph{Marked surface}]\label{marked surf def}
The term \newword{marked surface} refers to both nondegenerate and degenerate marked surfaces. 
\end{definition}

\begin{remark}\label{degen digon monogon}
The degenerate marked surfaces can be thought of as replacements for the monogon (disk with one marked point) and digon (disk with two marked points) that were explicitly ruled out in Definition~\ref{nondegen def}.
Since the digon's two boundary segments are isotopic to each other in the digon, we replace them, and the digon between them, with a single boundary segment.
Similarly, since the monogon's only boundary segment can be deformed to a point in $\S$, we replace the monogon with a point.
\end{remark}

\begin{definition}[\emph{Ambient isotopy in $\S$}]\label{ambient def}
Two subsets of $\S$ are related by \newword{ambient isotopy} if they are related by a homeomorphism from $\S$ to itself, fixing the boundary $\partial\S$ pointwise and homotopic to the identity by a homotopy that fixes $\partial\S$ pointwise at every step.
For brevity, we often refer to this notion simply as ``isotopy''.
\end{definition}

\begin{definition}[\emph{Arc}]\label{arc def}
An \newword{arc} in $(\S,\M)$ is a non-oriented curve in $\S$, with endpoints in $\M$, that
\begin{itemize}
\item
does not intersect itself except possibly at its endpoints, 
\item
except for its endpoints, does not intersect the boundary of $\S$,
\item
does not bound a monogon in $\S$, and
\item 
does not combine with a boundary segment to bound a digon in $\S$.
(This exclusion also applies to the case where the endpoints of the arc coincide, so that the digon has its two vertices identified.)
\end{itemize}
In particular, a degenerate marked surface has no arcs.
Arcs are considered up to ambient isotopy.
\end{definition}

\begin{definition}[\emph{Ring}]\label{ring def}
A closed curve is \newword{trivial} if it bounds a disk in $\S$.
A \newword{ring} is a nontrivial closed curve in $(\S,\M)$ that has no self-intersections, is disjoint from the boundary of  $\S$, and does not, together with an \emph{empty} boundary component of~$\S$, bound an annulus in $\S$.
Rings are considered up to ambient isotopy.
\end{definition}

\begin{definition}[\emph{Embedded block}]\label{embedded def}
An \newword{embedded block} in $(\S,\M)$ is a closed subset~$E$ of~$\S$, intersecting~$\M$ in a nonempty, finite set $\M_E$ such that $(E,\M_E)$ is a (degenerate or nondegenerate) marked surface, with the following restrictions:
\begin{itemize}
\item
If $(E,\M_E)$ is degenerate, it is either a point in $\M$ or an arc or boundary segment of $(\S,\M)$ with two distinct endpoints.
\item
If $(E,\M_E)$ is nondegenerate, each of its boundary components is a ring in $(\S,\M)$, an empty boundary component of $(\S,\M)$, or a finite union of arcs and/or boundary segments of $(\S,\M)$.
\item
The rings of $(\S,\M)$ occurring as boundary components of $E$ are distinct up to isotopy in $\S$.
\end{itemize}
Embedded blocks are considered up to ambient isotopy.
\end{definition}

\begin{definition}[\emph{Noncrossing partition}]\label{nc def}
A \newword{noncrossing partition} of $(\S,\M)$ is a collection $\P=\set{E_1,\ldots,E_k}$ of \emph{disjoint} embedded blocks (for some $k$ with ${1\le k\le|\M|}$) such that every point in $\M$ is contained in some (necessarily exactly one) $E_i$ and no two distinct blocks have the same ring (up to isotopy) in their boundary.
The embedded blocks $E_1,\ldots,E_k$ of $\P$ are called simply the \newword{blocks of~$\P$}.
Noncrossing partitions are considered up to ambient isotopy.
Thus a noncrossing partition of $(\S,\M)$ is a set partition of $\M$ together with the additional data of how each block is embedded.
A specific isotopy representative of a noncrossing partition~$\P$ is called an \newword{embedding} of~$\P$.
\end{definition}

Recall that Figure~\ref{nc ex fig} shows examples of noncrossing partitions of a sphere minus three open disks.

\begin{definition}[\emph{The noncrossing partition lattice}]\label{nc le def}
Noncrossing partitions $\P$ and~$\Q$ of $(\S,\M)$ have $\P\le\Q$ if and only if there exist embeddings of $\P$ and $\Q$ such that every block of $\P$ is contained in some block of $\Q$.
We can equivalently fix an embedding of $\P$ and say that $\P\le\Q$ if there exists an embedding of $\Q$ such that every block of $\P$ is contained in some block of $\Q$.
Or, by the same token, we can fix $\Q$ instead.
The symbol $\NC_{(\S,\M)}$ stands for the set of noncrossing partitions with this partial order.
Looking ahead to Theorem~\ref{lattice}, we call $\NC_{(\S,\M)}$ the \newword{noncrossing partition lattice} of $(\S,\M)$.
The symbol $\covered$ denotes cover relations in $\NC_{(\S,\M)}$.
\end{definition}

We now state our two main results on $\NC_{(\S,\M)}$.

\begin{thm}\label{lattice}
$\NC_{(\S,\M)}$ is a complete lattice.
\end{thm}

Our second main result shows that $\NC_{(\S,\M)}$ is graded and describes its rank function in terms of the topology of the union of the blocks of a noncrossing partition~$\P$.
This is a disjoint union of compact oriented surfaces with boundary together with (non-closed) curves and points.
The zeroth Betti number $b_0(\P)$ is the number of blocks of $\P$ and the first Betti number $b_1(\P)$ is the rank of the first homology (or less formally, the number of circles in $\P$ that don't bound disks in~$\P$).

\begin{theorem}\label{graded}
$\NC_{(\S,\M)}$ is graded, with rank function given by 
\[\rank(\P)=|\M|+b_1(\P)-b_0(\P).\]
\end{theorem}


As part of the proof of Theorem~\ref{graded}, in Proposition~\ref{covers} we characterize the cover relations in $\NC_{(\S,\M)}$.
Going up by a cover means replacing a noncrossing partition $\P$ with a noncrossing partition $\P\cup\alpha$, called the augmentation of $\P$ along $\alpha$, for a special choice of curve $\alpha$ called a simple connector for $\P$.
(See Definitions~\ref{curve union} and~\ref{simp conn} and Proposition~\ref{covers}.)

Before proving the two main theorems, we also point out the following fact, which follows immediately from the definitions, keeping in mind that an embedded block is itself a (possibly degenerate) marked surface.

\begin{prop}\label{lower}
If $\P$ is a noncrossing partition of $(\S,\M)$ with blocks $E_1,\ldots,E_k$, then the interval below $\P$ in $\NC_{(\S,\M)}$ is isomorphic to $\prod_{i=1}^k\NC_{(E_i,\M\cap E_i)}$.
\end{prop}

We now proceed to prove Theorem~\ref{graded} followed by Theorem~\ref{lattice}.
First, we establish the most basic property of $\NC_{(\S,\M)}$.

\begin{lemma}\label{poset}
$\NC_{(\S,\M)}$ is a bounded poset.
\end{lemma}
\begin{proof}
Reflexivity and transitivity of the relation are obvious.
For anti-symmetry, suppose $\P\le\Q\le\P$.
Then in particular, $\P$ and $\Q$ determine the same set partition of $\M$.
We will show that the blocks of $\P$ and $\Q$ are the same.
Let $E$ be a block of $\P$, let $E'$ be the block of $\Q$ containing $E$, and let $E''$ be the block of (another representative of) $\P$ containing $E'$.
If $E$ is degenerate, then it is apparent that $E=E'$, so suppose $E$ is nondegenerate.
Every boundary segment or empty boundary component of $\S$ that is part of the boundary of~$E$ is also part of the boundary of $E'$.
Every arc that is part of the boundary of~$E$, together with the corresponding arc for~$E''$, contain the arc for $E'$ between them, so that we can continuously deform this part of the boundary of $E$ to agree with $E'$.
Similarly, each ring in the boundary of $E$, together with the corresponding ring in the boundary of $E''$, contain the ring for $E'$ between them, so we can deform the ring for $E$ to agree with the ring for $E'$.
We see that $E$ and $E'$ are isotopic.

The poset $\NC_{(\S,\M)}$ has an element that is greater than all others, namely the noncrossing partition with precisely one block, which is $\S$ itself.
It also has an element that is less than all others, the partition with each element of $\M$ being a degenerate block.
\end{proof}

We next describe noncrossing partitions in terms of certain curves that they contain.

\begin{definition}[\emph{Curve set of a noncrossing partition}]\label{curve set def}
The \newword{curve set} $\curve(E)$ of an embedded block~$E$ is the union of two sets:  the set of boundary segments of $(\S,\M)$ that are contained in $E$ and the set of arcs of $(\S,\M)$ that (up to isotopy) are contained in~$E$.
The \newword{curve set} $\curve(\P)$ of a noncrossing partition $\P$ of $(\S,\M)$ is the (necessarily disjoint) union of the curve sets of its blocks.
\end{definition}

\begin{prop}\label{curve set le}
Two noncrossing partitions $\P$ and $\Q$ of $(\S,\M)$ have $\P\le\Q$ if and only if $\curve(\P)\subseteq\curve(\Q)$.
\end{prop}

\begin{proof}
One direction is immediate.
For the harder direction, suppose $\curve(\P)\subseteq\curve(\Q)$.
We will show that $\P\le\Q$ by showing that, for any block $E$ of $\P$, there is a block of $\Q$ that contains an embedding of $E$.
If $E$ is degenerate, this is immediate, so assume that $E$ is nondegenerate.
Since $\curve(\P)\subseteq\curve(\Q)$ and since the marked points in $E$ are connected by curves in $\curve(\P)$, all of the marked points in $E$ are contained in some block $E'$ of $\Q$.
Furthermore, every curve in $\curve(E)$ is necessarily in $\curve(E')$, rather than the curve set of some other block. 
 
%
The boundary of $E$ is defined by a finite set of boundary segments, arcs, and rings.
For each ring $U$ in the boundary of $E$, take a marked point $p$ in $E$ and create a curve $\alpha_U\in\curve(E)$ by following some path in $E$ from $p$ to a point very close to $U$, around $U$, and the back near the same path back to $p$.
(If $\alpha_U$, together with a boundary segment, bounds a digon in $\S$, then replace $\alpha_U$ with the boundary segment.)
Since $\curve(E)\subseteq\curve(E')$, there is an embedding of $U$ that is contained in $E'$.
(Start with an embedding of $\alpha_U$ that is contained, except for its endpoints, in the interior of $E'$.
The embedding of $U$ follows $\alpha_U$ closely in the interior of $E'$.)
For the same reason, every boundary segment of $\S$ that is in the boundary of $E$ is contained in $E'$.
Again for the same reason, we can choose isotopy representatives of the arcs of in the boundary of $E$, all of which are contained in $E'$, disjoint from each other except at endpoints.
Thus we can choose an isotopy representative of $E$ so that its boundary is contained in $E'$.

Suppose $E$ is not contained in $E'$.
Since the boundary of $E$ is in $E'$, the boundary of the set $E\cap E'$ is the union (necessarily disjoint) of the boundary of $E$ with some components $C_1,\ldots,C_k$ (with $k\ge 1$) of the boundary of $E'$, with each $C_i$ contained in the interior of $E$.
The interior of $E$ is contained in the interior of $\S$, so in particular it contains no marked points.
Therefore each $C_i$ is a ring. 
Furthermore, since no marked point is in the interior of $E$ and since the boundary of $E$ is contained in $E'$, if we pass from out of $E'$ through some $C_i$, we do not reach a marked point before re-entering $E'$ through some $C_j$ (possibly with $i=j$).
We see that the $C_i$ bound a subset of $\S$ without marked points that is topologically nontrivial.
Thus there is an arc $\beta$ that leaves $E'$ through some $C_i$, stays in $E$, and returns to $E'$ through some $C_j$ (possibly with $i=j$).
(Certainly there is such a non-self-intersecting curve disjoint from the boundary.
The nontrivial topology outside of the $C_i$ means that there is such a curve that does not bound a monogon or combine with a boundary segment to bound a digon.)

Since $\curve(E)\subseteq\curve(E')$, we see that $\beta$ is in $E'$, contradicting the construction of $\beta$ to cross boundary components of $E'$, with crossings that can't be removed by isotopy.
This contradiction proves that $E\subseteq E'$, as desired.
\end{proof}

\begin{prop}\label{curve set det}
A noncrossing partition of $(\S,\M)$ is determined uniquely (up to isotopy) by its curve set.
\end{prop}

\begin{proof}
If $\P$ and $\Q$ have the same curve set, then Proposition~\ref{curve set le} says that $\P\le\Q\le\P$.
The anti-symmetry in Proposition~\ref{poset} then implies that $\P=\Q$.
\end{proof}

\begin{remark}\label{equiv rel remark}
Proposition~\ref{curve set det} is the analog, for surfaces, of the easy and well known fact that set partitions determine equivalence relations and vice versa.
Whereas a set partition of $\M$ is determined by the set of pairs $p,q\in\M$ such that $p$ and $q$ are in the same block, a noncrossing partition $\P$ of $(\S,\M)$ is determined by the set of pairs $p,q\in\M$ such that $p$ and $q$ are in the same block of $\P$, together with the set of all ways that $p$ and $q$ are connected by arcs/boundary segments contained in that block.
\end{remark}

We now proceed with some definitions that will allow us to describe cover relations.

\begin{definition}[\emph{Curve union of embedded blocks}]\label{curve union}
Suppose $E$ and $E'$ are embedded blocks, either disjoint or equal.
Suppose $\alpha$ is an arc or boundary segment that has no isotopy representative inside $E$ or inside $E'$, but that has an isotopy representative that starts inside $E$, leaves $E$, then enters $E'$ and stays there until it ends there.
We write $E\cup\alpha\cup E'$ for the \newword{curve union of $E$ and $E'$ along $\alpha$}, which is the smallest embedded block containing $E$, $E'$, and $\alpha$, constructed as follows:
If $E$ and $E'$ are distinct points, then $E\cup\alpha\cup E'=\alpha$.
If $E=E'$ is a point, then $E\cup\alpha\cup E'$ is an annulus bounded by $\alpha$ and a ring that follows close to~$\alpha$.
In other cases, $E\cup\alpha\cup E'$ is essentially the union of $E$ and $E'$ with a ``thickened'' version of $\alpha$ (a suitably chosen digon bounded by $\alpha$ and arcs isotopic to $\alpha$, or bounded by $\alpha$ and a curve isotopic to $\alpha$ if $\alpha$ is a boundary segment).
If $E$ and/or $E'$ is a single arc or boundary segment, they might also need to be ``thickened'' in a similar way.
However, this union may fail to be an embedded block because one or more components of its boundary might not be a ring in $(\S,\M)$, an empty boundary component of $(\S,\M)$, or a finite union of arcs and/or boundary segments of $(\S,\M)$, or because the rings in its boundary might not be distinct up to isotopy.
We attach additional pieces of $\S$ to the union to correct these failures.
Specifically, if some component is a nontrivial closed curve in the interior of $\S$ that, together with an empty boundary component of $(\S,\M)$, bounds an annulus in $\S$, then we attach that annulus to the union.  
Also, if some component is composed of arcs and/or boundary segments and some other curves connecting marked points, such that each of these other curves, together with a boundary segment, bounds a digon, then we attach all such digons.
Finally, if two rings in the boundary are isotopic, then we attach the annulus between them to the union.
\end{definition}

\begin{example}\label{curve union ex}
Figure~\ref{curve union fig} shows two examples of curve unions, one example in each row of pictures.
In every picture, $\S$ is a torus with an open disk cut out and there are two marked points on the boundary.
The torus is a square with opposite sides identified as usual, but we have omitted the usual markings for identifying opposite sides.
In each case, the left picture shows the curve $\alpha$ (dotted) and the middle picture shows the union of a thickened $\alpha$ with the blocks it connects and also shows a striped area that must be added to form the curve union. 
The right picture shows the curve union.
In the first row, the striped area is a digon bounded by the thickened union and a boundary segment.
In the second row (which continues the first row), the striped area is an empty annulus between two components of the boundary of the thickened union, and the curve union is all of $\S$.

In both examples, the curve $\alpha$ connects a block to itself.
Later, in Example~\ref{aug ex} (the left and middle pictures of Figure~\ref{aug fig}), we illustrate the curve union in a case where $\alpha$ connects two distinct blocks.
\end{example}

\begin{figure}
\scalebox{1}{\includegraphics{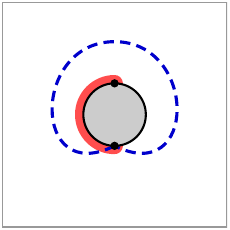}}\quad
\scalebox{1}{\includegraphics{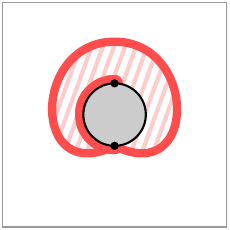}}\quad
\scalebox{1}{\includegraphics{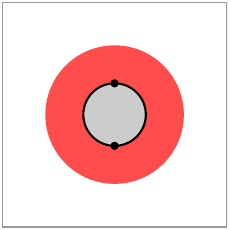}}\\[25pt]
\scalebox{1}{\includegraphics{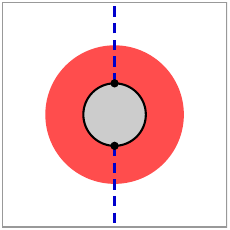}}\quad
\scalebox{1}{\includegraphics{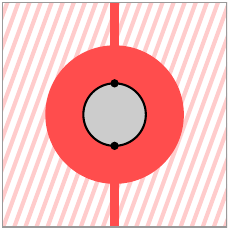}}\quad
\scalebox{1}{\includegraphics{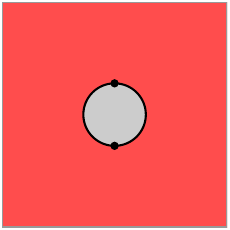}}
\caption{Examples of curve unions in a torus}
\label{curve union fig}
\end{figure}

\begin{definition}[\emph{Simple connector and augmentation}]\label{simp conn}
Suppose $\P$ is a noncrossing partition of $(\S,\M)$.
We say an arc or boundary segment $\alpha$ is a \newword{simple connector} for $\P$ if there are two embedded blocks $E$ and $E'$ of $\P$ (possibly with $E=E'$) such that $\alpha$ has no isotopy representative inside $E$ or inside $E'$, but has an isotopy representative that starts inside $E$, leaves $E$, \emph{intersects no other block of $\P$}, then enters $E'$ and stays there until it ends there.
If $\alpha$ is a simple connector for $(\S,\M)$, then the \newword{augmentation of $\P$ along $\alpha$}, written $\P\cup\alpha$, is essentially the noncrossing partition obtained by replacing $E$ and $E'$ by the curve union $E\cup\alpha\cup E'$, with the isotopy representative of~$\alpha$ chosen so that $E\cup\alpha\cup E'$ does not intersect any other blocks of $\P$.
However, the result of replacing $E$ and $E'$ by $E\cup\alpha\cup E'$ may fail to be a noncrossing partition because $E\cup\alpha\cup E'$ may have rings in its boundary that are also isotopic to rings in boundaries of other blocks $E''$ of $\P$.
In that case, the augmentation also adjoins annuli to combine $E\cup\alpha\cup E'$ with all such blocks $E''$.
\end{definition}

\begin{example}\label{aug ex}
The two examples of Figure~\ref{curve union fig} (Example~\ref{curve union ex}) also serve as examples of augmentation along a simple connector.  
Figure~\ref{aug fig} illustrates the additional step that may be neccessary to augment along a simple connector.  
The left picture (familiar from Figure~\ref{nc ex fig}) shows a noncrossing partition $\P$ of a sphere with three disks removed and a simple connector~$\alpha$, shown dotted.
The middle picture shows the noncrossing partition with two of its blocks replaced by their curve union along~$\alpha$.
The curve union has a ring in its boundary that is also in the boundary of another block of $\P$.
The annulus between the two blocks is striped.
The augmentation $\P\cup\alpha$ is the entire surface $\S$, as indicated in the right picture.
\end{example}

\begin{figure}
\scalebox{0.5}{\includegraphics{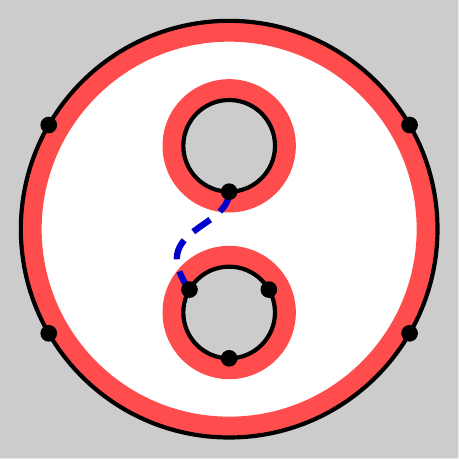}}\quad
\scalebox{0.5}{\includegraphics{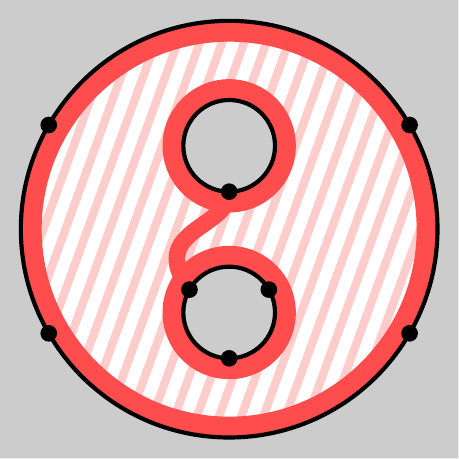}}\quad
\scalebox{0.5}{\includegraphics{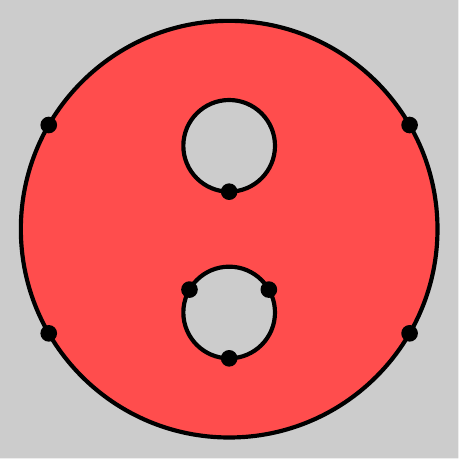}}
\caption{An example of augmentation along a simple connector}
\label{aug fig}
\end{figure}

\begin{lemma}\label{alpha does it}
Suppose $\P$ is a noncrossing partition of $(\S,\M)$ and suppose $\alpha$ is a simple connector for $\P$.
If $\R$ is a noncrossing partition of $(\S,\M)$ with $\P\le\R$ and $\alpha\in\curve(\R)$, then $\P\cup\alpha\le\R$.
\end{lemma}
\begin{proof}
Suppose $E$ and $E'$ are the two blocks (possibly coinciding) of $\P$ that are connected by $\alpha$.
Because $\P\le\R$ and $\alpha\in\curve(\R)$ and since $\alpha$ connects a marked point in $E$ to a marked point in $E'$, there is a block of $\R$ containing the curve union $E\cup\alpha\cup E'$.
By Definition~\ref{nc def}, the same block of $\R$ must also contain every block of $\P$ that has a boundary ring isotopic to a boundary ring of $E\cup\alpha\cup E'$.
Thus $\P\cup\alpha\le\R$.
\end{proof}

\begin{prop}\label{covers}
Two noncrossing partitions $\P,\Q\in\NC_{(\S,\M)}$ have $\P\covered\Q$ if and only if there exists a simple connector $\alpha$ for $\P$ such that $\Q=\P\cup\alpha$.
\end{prop}
\begin{proof}
First, suppose there exists a simple connector $\alpha$ for $\P$ such that $Q={P\cup\alpha}$ and let $E$ and $E'$ be the two blocks (possibly coinciding) of $\P$ that are connected by $\alpha$.
Certainly $\P<\Q$.
Suppose $\R$ is a noncrossing partition with $\P<\R\le\Q$.
By Lemma~\ref{alpha does it}, we can show that $\R=\Q$ by showing that $\alpha\in\curve(\R)$.

By Proposition~\ref{curve set det}, there exists a curve $\beta$ in $\curve(\R)\setminus\curve(\P)$.
If (up to isotopy) $\beta$ follows along the part of $\alpha$ not in $E$ or $E'$, then certainly $\alpha\in\curve(\R)$.
If $\beta$ does not follow along $\alpha$, then there are two similar possibilities:
Either $\beta$ passes through an annulus that was adjoined to the union of the thickened $\alpha$ with $E$ and $E'$ to remove isotopic boundary rings as part of the construction of the curve union, or~$\beta$ passes through an annulus that was adjoined to connect the curve union to an embedded block $E''$ with which it shared a boundary ring.
These two possibilities are argued together by allowing $E''=E$ or $E''=E'$ or both in what follows.
For either possibility, the thickened~$\alpha$ contains a part or parts of a ring $U$ such that the other part of $U$ is contained in $E\cup E'$ and such that a ring $U'$ isotopic to $U$ is a component of the boundary of $E''$.

If $E=E'$, then one part of $U$ is in the thickened $\alpha$, as exemplified in the middle picture of Figure~\ref{tworings fig}, where $U$ is the inner boundary of the striped annulus.
The curve $\beta$ passes through $U$ to leave $E$ and enters $E''$ through~$U'$.
The part of $\alpha$ that is outside $E$ (and $E''$) is isotopic to a curve contained in $\R$, which follows $\alpha$ to one intersection of $\alpha$ with the boundary of $E$, passes along the boundary of $E$, follows $\beta$ to $U'$, follows around $U'$, returns along $\beta$, and follows around the boundary of $E$ to the other intersection with $\alpha$, then follows $\alpha$ to its other endpoint, as illustrated in Figure~\ref{tworings fig}.
\begin{figure}
\scalebox{1}{\includegraphics{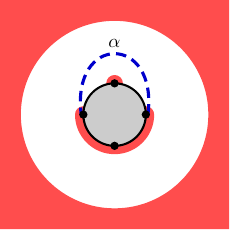}}\quad
\scalebox{1}{\includegraphics{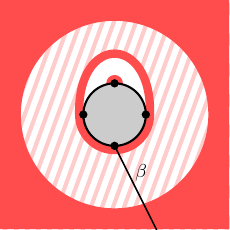}}\quad
\scalebox{1}{\includegraphics{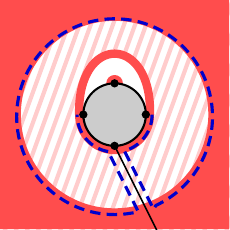}}
\caption{An illustration for the proof of Proposition~\ref{covers}}
\label{tworings fig}
\end{figure}
Thus $\alpha\in\curve(\R)$.

If $E\neq E'$, then two parts of $U$ are in the thickened $\alpha$ but not in $E\cup E'$, as exemplified in the left picture of Figure~\ref{gfig}, which shows the thickened $\alpha$ and the curve $\beta$, along with the rings $U$ (inner dotted) and $U'$ (outer dotted).
The curve $\beta$ passes through $U$ to leave (without loss of generality) $E$ and enters $E''$ through~$U'$.
In this case, we construct two isotopic rings that are contained in blocks of $\R$ and that are connected by $\alpha$.
Since $U$ passes twice along the thickened $\alpha$, we can remove those passes along $\alpha$ to make two rings contained respectively in the boundaries of $E$ and $E'$, shown (along with $U'$ again) in the middle picture of Figure~\ref{gfig}.
Then, using $\beta$, we can combine the ring in $E$ with the ring in $U'$ to create a ring $U''$ isotopic to the ring in $E'$.
The ring $U''$, together with the isotopic ring in $E'$, is shown in Figure~\ref{gfig}.
\begin{figure}
\scalebox{0.5}{\includegraphics{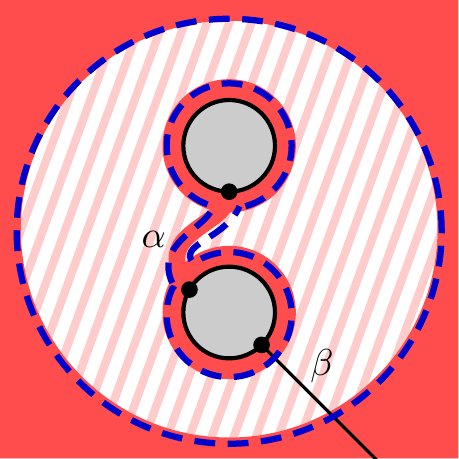}}\quad
\scalebox{0.5}{\includegraphics{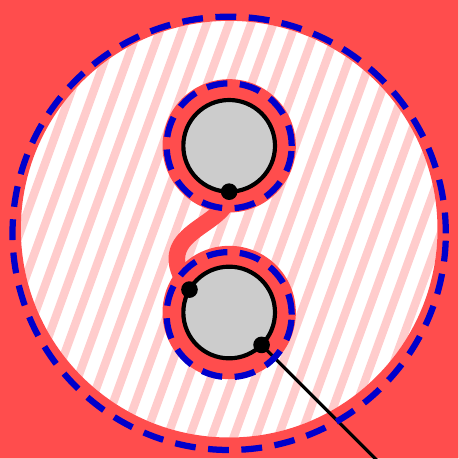}}\quad
\scalebox{0.5}{\includegraphics{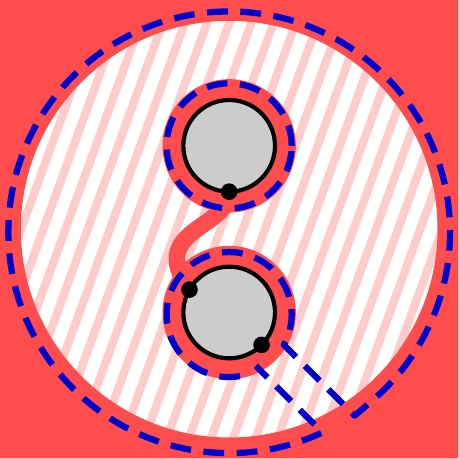}}
\caption{Another illustration for the proof of Proposition~\ref{covers}}
\label{gfig}
\end{figure}
Since $U''$ and the ring in $E'$ are isotopic and are in blocks of $\R$, they are in the same block, along with the annulus between them.
Thus $\alpha\in\curve(\R)$.

We see that $\alpha\in\curve(\R)$ in every case, and we have proved that $\P\covered\Q$.

Conversely, suppose $\P\covered\Q$.
Then Proposition~\ref{curve set det} says in particular that there exists a curve $\beta\in\curve(\Q)\setminus\curve(\P)$.
Choose representatives of $\P$ and $\Q$ such that every block of $\P$ is contained in a block of $\P$ and choose a representative of~$\beta$ that is contained in a block of $\Q$.
Since $\beta\not\in\curve(\P)$, the representative of~$\beta$ is not contained in a block of $\P$.
Let $E$ be the block containing one endpoint of~$\beta$.
Let $E'$ be the first block of $\P$ that $\beta$ enters after first leaving $E$ (not ruling out the possibility that $E'=E$), and assume that the intersection of $\beta$ with $E'$ cannot be removed by choosing a different isotopy representative of~$\beta$.
Let $\alpha$ be a curve that follows $\beta$ from $E$ into $E'$ and then stays in $E'$ to end at some marked point in $E'$.
(Possibly, $\alpha=\beta$.)
Then $\alpha$ is a simple connector for $\P$, so that $\P\covered \P\cup\alpha$.
Lemma~\ref{alpha does it} implies that $\P\cup\alpha\le\Q$, but since $\P\covered\Q$, we conclude that $\Q=\P\cup\alpha$.
\end{proof}

\begin{proof}[Proof of Theorem~\ref{graded}]
The formula for the rank function is correct when $\P$ is the bottom element (each marked point is a singleton block), so we need to verify that going up by a cover relation adds $1$ to $b_1(\P)-b_0(\P)$.
Suppose $\alpha$ is a simple connector for $\P$, connecting blocks $E$ and $E'$.

First, we compute the change in $b_1-b_0$ when replacing $E$ and $E'$ by $E\cup\alpha\cup E'$:
If $E\neq E'$, then $b_1$ is unchanged and $b_0$ decreases by exactly $1$.
If $E=E'$, then $b_0$ is unchanged, while $b_1$ increases by exactly $1$, with the new circle in $E\cup\alpha\cup E'$ starting in $E$ and following $\alpha$ back to $E$.

Next, we compute the effect of joining $E\cup\alpha\cup E'$ to another block or to itself because of shared boundary rings.
If $E\cup\alpha\cup E'$ is joined to itself, then $b_0$ does not change, and $b_1$ loses $1$ (because the boundary ring is only counted once after the joining) but also gains $1$ (the new circle where $E\cup\alpha\cup E'$ connects to itself).
If $E\cup\alpha\cup E'$ is joined to another block, then $b_0$ is decreased by 1, but also $b_1$ decreases by $1$ because the boundary ring is counted twice (once in each block) before the blocks are combined but only once in the combined block.
In either case, $b_1-b_0$ is unchanged.

In all, the quantity $b_1-b_0$ increases by exactly $1$ when passing from $\P$ to $\P\cup\alpha$.
\end{proof}

Our next goal is to prove Theorem~\ref{lattice}, which says that $\NC_{(\S,\M)}$ is a complete lattice.

\begin{definition}[\emph{Compatible arcs}]\label{compat arc def}
Two arcs are \newword{compatible} if there exists an isotopy representative of each such that the representatives do not intersect, except possibly at their endpoints.
Maximal pairwise compatible sets of arcs are called \newword{triangulations}.
\end{definition}

We do not need triangulations in this paper, except for the fact that all triangulations of $(\S,\M)$ have the same finite cardinality.
When there are no empty boundary components, as a special case of \cite[Proposition~2.10]{cats1}, a triangulation has $6g+3b+|\M|-6$ arcs, where $b$ is the number of boundary components and $g$ is the genus of~$\S$.
(By convention, $g$ is the genus of the surface without boundary obtained by filling in all boundary components with disks.)
The following is a minor modification of \cite[Proposition~2.10]{cats1}.

\begin{proposition}\label{tri}
If $(\S,\M)$ has genus $g$ and has $b$ boundary components, $e$ of which are empty, then a triangulation of $(\S,\M)$ has $6g+3b-e+|\M|-6$ arcs.
\end{proposition}
\begin{proof}
This holds by an easy induction on $e$, with \cite[Proposition~2.10]{cats1} serving as the base case.
If $U$ is an empty boundary component of $\S$, then any triangulation~$T$ of $(\S,\M)$ contains an arc that starts at some marked point $m$, follows a curve towards $U$, goes around $U$, and follows near the same curve back to $m$.
We see that if $\M'$ is obtained from $\M$ by putting a marked point on $U$, then we can obtain a triangulation $T'$ of $(\S,\M')$ by adjoining two arcs to $T$, both connecting~$m$ to the new marked point on $U$.
By induction on $e$, the number of arcs in $T'$ is $6g+3b-(e-1)+(|\M|+1)-6$, so $T$ has $6g+3b-e+|\M|-6$ arcs.
\end{proof}

\begin{definition}[\emph{Concatenation of arcs/boundary segments}] \label{concat def}
Suppose $\alpha$ and $\beta$ are distinct (up to isotopy) compatible arcs or boundary segments in $(\S,\M)$, intersecting in at least one endpoint $p$.
(For the purposes of this definition, any two distinct boundary segments are compatible and any arc is compatible with any boundary segment.)
A \newword{concatenation} of $\alpha$ and $\beta$ is an arc or boundary segment~$\gamma$ obtained by following close to $\alpha$ from an endpoint towards $p$, until getting close to~$p$, and then following close to $\beta$ to an endpoint.
If $\beta$ has both endpoints at $p$, then it is possible that there are two concatenations obtained from following $\alpha$ then $\beta$ (corresponding to the two directions to follow $\beta$ from $p$).
It is also possible that there is only one concatenation, because only one direction can yield a non-self-intersecting curve.  
Similar conditions apply if $\alpha$ has both endpoints at $p$.
(This issue is whether one can pass from $\alpha$ to $\beta$, close to~$p$, without crossing $\alpha$ or $\beta$.
If not, then the resulting curve has self-intersections.
This depends on the order in which the ends of $\alpha$ and $\beta$ are arranged around $p$.)
A concatenation $\gamma$ is an arc unless it combines with a boundary segment to bound a digon, in which case, we take $\gamma$ to be that boundary segment.
It is impossible for $\gamma$ to bound a monogon in~$\S$, because $\alpha$ and~$\beta$ are distinct up to isotopy.
If $\alpha$ and $\beta$ are incompatible arcs or do not share an endpoint, then they have no concatenations.
\end{definition}

\begin{definition}[\emph{Resolution of arcs}] \label{res arc def}
Suppose $\alpha$ and $\beta$ are incompatible arcs.
Choose isotopy representatives of $\alpha$ and $\beta$ that avoid unnecessary intersections.
(Specifically, if there are two intersection points of $\alpha$ with $\beta$ such that the part of~$\alpha$ between the two intersections and the part of $\beta$ between the two intersections together bound a disk in $\S$, then choose isotopy representatives that remove these intersections.)
A \newword{resolution} of $\alpha$ and $\beta$ is a non-self-intersecting curve $\gamma$ obtained by following $\alpha$ from an endpoint of $\alpha$ to an intersection of $\alpha$ and $\beta$ and then following $\beta$ to an endpoint.
If we try to make a resolution of $\alpha$ and $\beta$ at a given intersection, it is possible that the result has self-intersections, and thus is not a resolution.
However, if we follow $\alpha$ from an endpoint to its \emph{first} intersection with $\beta$ and resolve there, the result is always a non-self-intersecting curve.

If $\gamma$ fails to be an arc because it combines with a boundary segment to bound a digon, then take $\gamma$ to be that boundary segment.
It is impossible for $\gamma$ to bound a monogon in $\S$ because representatives of $\alpha$ and $\beta$ were chosen to avoid unnecessary intersections.

If $\alpha$ and $\beta$ are not incompatible arcs, then they have no resolutions.
\end{definition}

\begin{prop}\label{closure}
For any noncrossing partition $\P$ of $(\S,\M)$, the set $\curve(\P)$ is closed under concatenations and resolutions, meaning that if $\alpha$ and $\beta$ are in $\curve(\P)$, then any concatenation or resolution of $\alpha$ and $\beta$ is in $\curve(\P)$.
\end{prop}

\begin{proof}
Suppose $\P\in\NC_{(\S,\M)}$.
Since distinct blocks of $\P$ are disjoint, there are no concatenations or resolutions involving arcs/boundary segments from different blocks.
Thus, to show that $\curve(\P)$ is closed, it is enough to show that $\curve(E)$ is closed for each block $E$ of~$\P$.

Suppose $\alpha$ and $\beta$ are distinct arcs/boundary segments in $\curve(E)$ meeting at a point~$p$.
Then in particular, $E$ is nondegenerate, so near $p$, we can pass from $\alpha$ to $\beta$ through the interior of $E$ to construct an isotopy representative inside~$E$ for the contatenation of $\alpha$ and $\beta$.
If $\alpha,\beta\in\curve(E)$ are incompatible arcs, with representatives chosen to be in $E$, and $\gamma$ is a resolution of $\alpha$ and $\beta$, then $\gamma$ is also in $E$.
We have shown that $\curve(E)$ is closed as desired.
\end{proof}

\begin{defcon}[\emph{Embedded intersection of embedded blocks}]\label{embed int def}
We now define a notion of an \newword{embedded intersection} of embedded blocks $E_1$ and $E_2$.
An embedded intersection of $E_1$ and $E_2$ is yet another embedded block.
We say \emph{an} embedded intersection deliberately, because two blocks $E_1$ and $E_2$ may have more than one embedded intersection or may have none.

We continue the notation $\M_E=\M\cap E$ from Definition~\ref{embedded def} and we write $C$ to stand for ${\curve(E_1)\cap\curve(E_2)}$.
Throughout, we use the fact that $C$ is closed under concatenation and resolution.
(By Proposition~\ref{closure}, $C$ is an intersection of sets that are closed under concatenation and resolution.)

If $\M_{E_1}\cap\M_{E_2}$ is empty, then $E_1$ and $E_2$ have no embedded intersections.
If $\M_{E_1}\cap\M_{E_2}\neq\emptyset$, then we define an equivalence relation on $\M_{E_1}\cap\M_{E_2}$ by setting distinct marked points $p$ and $q$ to be equivalent if and only if there exists $\alpha\in C$ with endpoints $p$ and~$q$.
We construct exactly one embedded intersection of $E_1$ and $E_2$ containing each equivalence class in $\M_{E_1}\cap\M_{E_2}$.

Let $X$ be an equivalence class.
We will construct an embedded intersection $E_X$ of $E_1$ and $E_2$ such that $M_{E_X}=X$.
Write $C_X$ for the set of curves in $C$ that have endpoints in $X$.
If $|C_X|=0$, then $X$ is a singleton and $E_X$ is $X$.
If $|C_X|=1$ and $|X|=2$, then $E_X$ coincides with the unique curve in $C$.
If $|C_X|=1$ and $|X|=1$, then $E_X$ is the annulus bounded by the unique curve in $C$ together with the ring in $\S$ that follows that curve. 
Suppose now that $|C_X|>1$.
We first find a finite collection of curves in $C_X$ that, together with some rings and empty boundary components, will serve as the boundary of~$E_X$.

Given $p\in X$, we use the words left and right to describe arcs incident to~$p$.
That is, if $\alpha$ and $\beta$ are arcs/boundary segments in $C_X$ with endpoints at $p$, then one exits $p$ further to the left than the other (from the viewpoint of $p$), whenever isotopy representatives are chosen to avoid unnecessary intersections.
(If one of these curves has both endpoints at $p$, then it exits $p$ twice, one time to the left of the other, and possibly with other arcs exiting in between.)
We will find the \newword{leftmost curve} in $C_X$ incident to $p$:  a curve in $C_X$ that exits to the left of every other curve in $C_X$ incident to $p$.

If the boundary segment exiting $p$ to the left is in $C_X$, then it is the leftmost curve.
Otherwise, if there is no arc in $C_X$ incident to $p$, then $|C_X|\le 1$, but we have already dealt with that case.
Suppose the boundary segment left from $p$ is not in $C_X$ and suppose $\alpha_0$ is a curve in $C_X$ incident to $p$.
If $\alpha_0$ is not the leftmost curve incident to $p$, then there is an arc $\alpha_1$ that is compatible with $\alpha_0$ and is left of~$\alpha_0$.
To find $\alpha_1$, take an arc $\beta$ left of $\alpha_0$.
If $\beta$ is compatible with $\alpha_0$, then we take~$\alpha_1=\beta$, and otherwise we obtain $\alpha_1$ as a resolution of $\beta$ and $\alpha_0$.
If $\alpha_1$ is not the leftmost curve incident to $p$, then there is an arc $\alpha_2$ that is compatible with $\alpha_0$ and $\alpha_1$ and is left of $\alpha_1$.
To find $\alpha_2$, take an arc $\beta$ left of $\alpha_1$.
If $\beta$ is compatible with $\alpha_0$ and $\alpha_1$, then we can take $\alpha_2=\beta$.
Otherwise, find the first point along $\beta$ that intersects $\alpha_0$ or $\alpha_1$ and let $\alpha_2$ be the resolution of $\beta$ and $\alpha_0$ or $\alpha_1$ at that point.
Then $\alpha_2$ is compatible with $\alpha_0$ and $\alpha_1$ (since $\alpha_0$ and $\alpha_1$ are compatible) and is left of~$\alpha_1$.
We continue finding successive arcs $\alpha_i$, with each $\alpha_i$ left of $\alpha_{i-1}$ and compatible with each of $\alpha_1,\ldots,\alpha_{i-1}$ until we find the leftmost arc incident to $p$.
We eventually find the leftmost arc because Proposition~\ref{tri} is a bound on the size of a pairwise compatible set of arcs, depending only on $\S$ and $\M$.

We now make two observations about the leftmost curve $\alpha$ in $C_X$ incident to~$p$.
First, if $q$ is the other endpoint of $\alpha$ (not ruling out the possibility that $q=p$), then $\alpha$ is the rightmost curve in $C_X$ incident to $q$.
(If there is some curve $\beta$ at $q$ right of $\alpha$, then concatenating $\alpha$ and $\beta$ yields a curve in $C_X$ that is left of $\alpha$ at~$p$.)
Second, if $\alpha$ is an arc, then every other arc in $C_X$ is compatible with $\alpha$.
(If some arc $\beta\in C_X$ is incompatible with $\alpha$, then there is a resolution of $\alpha$ and $\beta$, following $\alpha$ to its first intersection point with $\beta$ and then turning left onto $\beta$ to produce a curve in $C_X$ that is left of $\alpha$ at $p$.)

We now choose isotopy representatives of $E_1$ and $E_2$ and an isotopy representative of the leftmost curve of every point $p\in X$, subject to the following conditions:
no two of the representatives intersect each other, except possibly at endpoints;
each leftmost curve is contained in $E_1$ and (unless it is a boundary segment of $\S$) disjoint from boundary of $E_1$ except at endpoints;
the same condition applies for $E_2$; and if $U_1$ is an arc or ring that forms part of the boundary of $E_1$ and $U_2$ is an arc or ring that forms part of the boundary of $E_2$, then $U_1$ and $U_2$ are chosen to avoid unnecessary intersections (in the same sense described in Definition~\ref{res arc def}).

Let $L_X$ stand for the union, over all $p\in X$ of these representatives of leftmost curves in $C_X$ incident to $p$.
Each such curve is the unique leftmost curve for one of its endpoints and the unique rightmost curve for the other, so $L_X$ is a disjoint union of non-self-intersecting closed curves.

A given closed curve in $L_X$ can fail to separate $\S$ into two pieces, but it has, locally, two sides.
The \emph{inside} of the curve is the side on the \emph{right} when we move from a point $p\in X$ along the leftmost curve at $p$.
The other side is the \emph{outside}.

We finally can construct $E_X$:
It is the subset of $\S$ consisting of all the points of~$\S$ that can be the endpoint of a path starting on a curve in $L_X$, leaving that curve towards the inside, and not crossing any curves in $L_X$ or the boundary of $E_1$ or $E_2$ (but possibly ending on one of these boundaries).
We see that $E_X$ is a subset of the chosen representative of $E_1$ and a subset of the chosen representative of $E_2$.
\end{defcon}

\begin{prop}\label{embed int nc}
If $E_1$ and $E_2$ are embedded blocks in $(\S,\M)$, then any embedded intersection $E$ of $E_1$ and $E_2$ is an embedded block in $(\S,\M)$.
If $\M_E$ is the set of marked points in $E$, then $\curve(E)$ is the subset of $\curve(E_1)\cap\curve(E_2)$ consisting of curves whose endpoints are in $\M_E$.
\end{prop}
\begin{proof}
If $E$ is a point or a curve with two distinct endpoints, then $E$ is a degenerate marked surface.
By construction, when $E$ is a point, it is in $\M$, and when $E$ is a curve, it is an arc or boundary segment of $(\S,\M)$, so $E$ is an embedded block in either case.

If $E$ is not a point or curve, it inherits an orientation as a subset of $\S$.
By construction, $E$ has as its boundary a collection of circles, as follows:
First, there are the curves whose union is $L_X$ in Definition/Construction~\ref{embed int def} (with $X=\M_E$).
Next there are any empty boundary components of $(\S,\M)$ that are in the boundary of $E_1$ and in the boundary of $E_2$.
Finally, there are some rings in the boundary of~$E_1$, some rings in the boundary of $E_2$, and/or some rings in $\S$ formed by following a ring in the boundary of $E_1$, switching at the intersection to follow a ring in the boundary of $E_2$, back to $E_1$, and so forth.
Thus in this case also, $E$ is an embedded block.

Since $E$ is a subset of some representative of $E_1$ and a subset of some representative of $E_2$, we see that $\curve(E)\subseteq\curve(E_1)\cap\curve(E_2)$.
Since $\M_E$ is defined as an equivalence class under the equivalence relation of being connected by a curve in $\curve(E)\subseteq\curve(E_1)$, each curve in $\curve(E)$ connects points in $
\M_E$.

On the other hand, suppose $\alpha\in\curve(E_1)\cap\curve(E_2)$ has its endpoints in $\M_E$.
Then there is an isotopy representative of $\alpha$ in the isotopy representative of $E_1$ that was chosen in Definition/Construction~\ref{embed int def}.
Similarly, there is a representative of~$\alpha$ in the representative of $E_2$.
Each representative can be chosen to not cross any of the leftmost curves $\gamma$ that make up $L_X$, because otherwise, a resolution of $\alpha$ and $\gamma$ would be left of $\gamma$.

We want to construct a single isotopy representative of $\alpha$ that is contained in both $E_1$ and $E_2$.
We can do this by following the two representatives in the same direction and making a curve isotopic to both as we go.  
We are unable to continue at a point where one representative leaves $E_1$ and the other leaves $E_2$.
This happens at an intersection of the boundary of $E_1$ with the boundary of $E_2$.
Following the two curves separately until they meet again, the two curves trace out a disk.
But for them to meet again, they must both be in $E_1$ and in $E_2$.
Thus that disk encloses intersections of the boundaries of $E_1$ and $E_2$ that can be removed by isotopy, contradicting the choice of $E_1$ and $E_2$ so that their boundaries avoid unnecessary intersections.
By this contradiction, we see that there is an isotopy representative of $\alpha$ that is contained in $E_1$ and in $E_2$, and thus inside $E$.
Thus $\alpha\in\curve(E)$.
\end{proof}

\begin{prop}\label{meet prop}
Suppose $\P$ and $\Q$ are noncrossing partitions of $(\S,\M)$.
Let~$\R$ be the set of all embedded blocks that arise as an embedded intersection of a block of~$\P$ with a block of $\Q$.
Then $\R$ is a noncrossing partition of $(\S,\M)$ with $\curve(\R)=\curve(\P)\cap\curve(\Q)$.
\end{prop}
\begin{proof}
The key to the proof is to show that $\R$ is a noncrossing partition.
If so, since $\curve(\R)$ is the union of the curve sets of its blocks, Proposition~\ref{embed int nc} says that $\curve(\R)$ is the union, over all blocks $E_1$ of $\P$ and $E_2$ of $\Q$, of $\curve(E_1)\cap\curve(E_2)$.
This union is $\curve(\P)\cap\curve(\Q)$.

We now show that $\R$ is a noncrossing partition.
For each pair $(E_1,E_2)$ consisting of a block $E_1$ of $\P$ and a block $E_2$ of $\Q$, we constructed an equivalence relation on $\M_{E_1}\cap\M_{E_2}$.
The sets $\M_{E_1}\cap\M_{E_2}$, as $(E_1,E_2)$ varies, constitute a set partition of $\M$, so the set of all equivalence classes arising for all pairs $(E_1,E_2)$ defines an equivalence relation on $\M$.
Proposition~\ref{embed int nc} implies that every curve in $\curve(\P)\cap\curve(\Q)$ connects marked points in the same equivalence class $X$.
The leftmost curves (constructed in Definition/Construction~\ref{embed int def}), as $X$ varies, are all compatible with each other because if two were incompatible, they would admit as a resolution a curve connecting distinct equivalence classes.
Thus we can choose representatives of the various closed curves $L_X$ to be disjoint from each other.

We then embed any other components of the boundaries of all blocks of $\P$ and~$\Q$ so as to avoid unnecessary intersections between curves in the boundary of $\P$ and curves in the boundary of $\Q$.
Constructing all the embedded intersections using these fixed embeddings of blocks of $\P$ and blocks of $\Q$, we obtain pairwise disjoint embeddings of the embedded intersections.

We have almost shown that $\R$ is a noncrossing partition.
It remains to show that no two distinct blocks of $\R$ have the same ring in their boundary.
Suppose to the contrary that for blocks $E_1,E_1'$ of $\P$ and blocks $E_2,E_2'$ of $\Q$ there exist an embedded intersection $E_3$ of $E_1$ and $E_2$ and an embedded intersection $E_3'$ of $E_1'$ and~$E_2'$ that are distinct but have a boundary ring in common.
Let $U$ and $U'$ be the two embeddings of that ring that bound $E_3$ and $E_3'$.
Since $E_3$ is in $E_1$ and~$E_2$, also~$U$ is in $E_1$ and $E_2$, and since $E'_3$ is in $E'_1$ and $E'_2$, also~$U'$ is in $E'_1$ and $E'_2$.
Since~$\P$ is a noncrossing partition, $E_1$ and $E_1'$ cannot both have boundary rings isotopic to~$U$ and $U'$, so $E_1=E_1'$, and this block goes through the annulus defined by $U$ and~$U'$.
But since the boundary of $E_1$ cannot have a component that is a trivial closed curve in $\S$, we conclude that $E_1$ contains the entire annulus defined by $U$ and $U'$.
Similarly, $E_2$ and $E'_2$ coincide and contain the entire annulus, contradicting that supposition that $E_3$ and $E_3'$ are separated by the annulus.
We conclude that $\R$ is a noncrossing partition.
\end{proof}

\begin{proof}[Proof of Theorem~\ref{lattice}]  
If $\P$ and $\Q$ are noncrossing partitions of $(\S,\M)$, then Propositions~\ref{curve set le} and~\ref{meet prop} together imply that the meet $\P\meet\Q$ exists and equals the noncrossing partition whose embedded blocks are all of the embedded intersections of blocks of $\P$ and blocks of $\Q$.
We have shown that $\NC_{(\S,\M)}$ is a meet semilattice.
Theorem~\ref{graded} implies in particular that $\NC_{(\S,\M)}$ has finite height, so we see that $\NC_{(\S,\M)}$ is a \emph{complete} meet semilattice.
Since $\NC_{(\S,\M)}$ has an upper bound, $\NC_{(\S,\M)}$ is a complete lattice by the standard argument (showing that the join is the meet of all upper bounds).
\end{proof}

\begin{remark}\label{not int closed}
We do not have a proof for the converse of Proposition~\ref{closure} (the assertion that every set of arcs/boundary curves that is closed under concatenation is $\curve(\P)$ for some noncrossing partition $\P$).
It is possible that one could construct a noncrossing partition for a given closed set of curves using some version of Definition/Construction~\ref{embed int def}.
We already have a characterization of the meet in $\NC_{(\S,\M)}$ in terms of curve sets:
$\curve(\P\meet\Q)=\curve(\P)\cap\curve(\Q)$.
If the converse of Proposition~\ref{closure} were known, then we could describe the join in terms of curve sets as well:
$\curve(\P\join\Q)$ would be the unique smallest closed set containing $\curve(\P)\cup\curve(\Q)$ (the intersection of all closed sets containing the union).
\end{remark}

We close this section by mentioning a corollary to Theorem~\ref{lattice}.
The curve set of a noncrossing partition $\P$ can be thought of as a collection of noncrossing partitions:
Each curve in $\curve(\P)$ corresponds to the noncrossing partition one of whose blocks is that curve or an annulus bounded by that curve and a ring, and whose other blocks are single points.
We abuse notation in the following theorem by thinking of $\curve(\P)$ as the set of all such noncrossing partitions.
The theorem is immediate in light of Proposition~\ref{curve set le} and Theorem~\ref{lattice}.

\begin{corollary}\label{join curve}
Suppose $\P\in\NC_{(\S,\M)}$.
Then $\P=\Join\curve(\P)$.
\end{corollary}

\section{Symmetric noncrossing partitions with double points}\label{doub sec}  
In this section, we define symmetric noncrossing partitions of a surface with marked points on its boundary and double (marked) points on its interior and prove the properties summarized in Section~\ref{intro sec}.
The noncrossing partitions are symmetric with respect to a fixed involutive symmetry.
The existence of double points and the requirement of symmetry work together to form a model that is distinct from the noncrossing partitions of a marked surface, but that uses some of the same definitions.
Other definitions must be modified in the presence of symmetry/double points.
This section is written to parallel Section~\ref{marked sec}, for easy comparison.

\begin{definition}[\emph{Symmetric marked surface with double points}]\label{sym doub def}
We start with a surface $\S$ and write $\partial\S$ for the boundary of $\S$ and $\int(\S)$ for the interior.
Informally, we ``double'' some finite set $\D\subset\int(\S)$:
For each point in $\D$, there are ``two copies'' of that point, distinguished with a ``$+$'' and a ``$-$''.
More formally, we consider the topological space~$\S^\pm$ obtained from two disjoint copies $\S^+$ and $\S^-$ of~$\S$ by identifying the two copies at every point except at each point in $\D$.
If $\D\neq\emptyset$, then $\S^\pm$ is not a surface (or even a Hausdorff space), but we will use standard surface terminology, which is understood to describe the underlying surface $\S$.

A \newword{symmetric marked surface with double points} $(\S^\pm,\B,\D^\pm,\phi)$, or $\S^\pm$ for short, is 
\begin{itemize}
\item a nondegenerate marked surface $(\S,\B)$ in the sense of Definition~\ref{nondegen def} (but with $\B$ allowed to be empty),
\item a finite set $\D$ of \newword{double points} in the interior of $\S$ that is used to define the space $\S^\pm$ containing two disjoint copies of $\D$ whose union is denoted~$\D^\pm$, and
\item an involutive orientation-preserving homeomorphism $\phi$ of $\S^\pm$ that permutes~$\B$ and sends each double point in $\D^\pm$ to the other double point \emph{at the same location}.
We also use the symbol $\phi$ for the involutive homeomorphism induced on $\S$.
We require that this induced map $\phi$ on $\S$ is a \emph{nontrivial} homeomorphism of $\S$.
This map thus fixes $\D$ pointwise and acts as a half-turn rotation near each $d\in\D$.
The map $\phi$ may also have additional isolated fixed points, and it acts as a half-turn rotation near each of these fixed points as well.
\end{itemize}
The points $\B$ are called \newword{marked boundary points} and the points $\D^\pm$ are called \newword{double points}.
The set $\B\cup\D^\pm$ is denoted $\M$ and is called the set of \newword{marked points}.
We require that $\M$ is not empty.
\end{definition}

\begin{remark}\label{why phi}
Double points are less obviously inspired by \cite{cats1,cats2}, which instead has punctures (non-doubled marked points in the interior of $\S$).
As discussed in Remark~\ref{clus rem}, the combination of double points and symmetry plays a role that one might have expected to be played by punctures.
\end{remark}

\begin{definition}[\emph{Symmetric ambient isotopy}]\label{sym iso def}
Two subsets of $\S^\pm$ are related by \newword{symmetric ambient isotopy} if they are related by a homeomorphism from $\S^\pm$ to itself, fixing the boundary $\partial\S$ pointwise, fixing $\D^\pm$ pointwise, commuting with~$\phi$, and homotopic to the identity by a homotopy that fixes $\partial\S$ and $\D^\pm$ pointwise at every step and commutes with $\phi$ at every step.
A symmetric ambient isotopy necessarily fixes every fixed point of $\phi$, so in particular it cannot relate subsets of~$\S^\pm$ by passing them through a fixed point of $\phi$.
We often refer to symmetric ambient isotopy simply as ``symmetric isotopy''.
\end{definition}

\begin{definition}[\emph{Arc in a symmetric marked surface with double points}]\label{sym arc def}
We modify Definition~\ref{arc def} to account for the presence of double points.
An \newword{arc} in $\S^\pm$ is a non-oriented curve $\alpha$ in $\S^\pm$, having endpoints in $\M$ and satisfying certain requirements.
Some of the requirements are familiar from Definition~\ref{arc def}:
\begin{itemize}
\item
$\alpha$ does not intersect itself except possibly at its endpoints.
\item
Except for its endpoints, $\alpha$ does not intersect $\M$ or the boundary of $\S$.
\item
$\alpha$ does not bound a monogon in $\S$. 
This exclusion applies even if the endpoints of $\alpha$ are two different double points at the same location in $\S$.
\item 
$\alpha$ does not combine with a boundary segment to bound a digon in $\S$. 
\end{itemize}
There are also two new requirements related to symmetry:
\begin{itemize}
\item
Either $\alpha=\phi(\alpha)$ or $\alpha$ and $\phi(\alpha)$ don't intersect, except possibly at endpoints.
\item
$\alpha$ and $\phi(\alpha)$ do not combine to form a digon in $\S^\pm$ unless that digon contains a point in $\D$.
The curves $\alpha$ and $\phi(\alpha)$ \emph{are} allowed to form a digon in $\S$ (as opposed to $\S^\pm$) if, at each vertex of the digon, the two edges at that vertex attach to different copies of the same double point.
Figure~\ref{arc digon} illustrates this distinction in a symmetric annulus with one double point and one non-double fixed point.
In the left picture, the curves $\alpha$ and $\phi(\alpha)$ both attach to the double point marked $+$ at one vertex of the digon and both attach to the double point marked $-$ at the other vertex, so $\alpha$ and $\phi(\alpha)$ are \emph{not} arcs.
By contrast, in the right picture, $\alpha$ and $\phi(\alpha)$ \emph{are} arcs.
\begin{figure}
\begin{tabular}{cc}
\scalebox{0.5}{\includegraphics{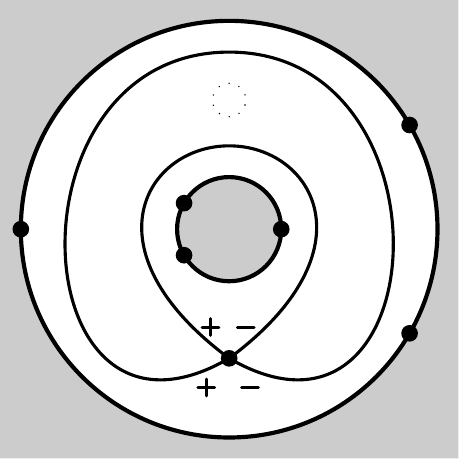}}&
\scalebox{0.5}{\includegraphics{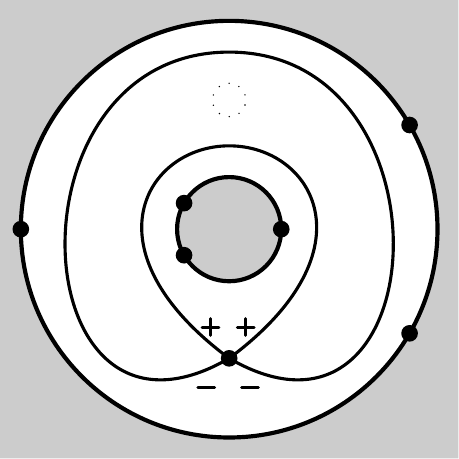}}
\end{tabular}
\caption{Ways that an arc pair may not (left) or may (right) form an empty digon}
\label{arc digon}
\end{figure}
\end{itemize}

When $\alpha=\phi(\alpha)$, we call $\alpha$ a \newword{symmetric arc} and otherwise $\alpha$ and $\phi(\alpha)$ form a \newword{symmetric pair of arcs}, or \newword{arc pair} for short.
A symmetric arc necessarily contains, in its relative interior, a point that is fixed by $\phi$.
Since an arc may not pass through marked points (including double points), symmetric arcs cannot exist in $(\S^\pm,\B,\D^\pm,\phi)$ unless the map $\phi$ fixes at least one point in $\S\setminus\D$.

Symmetric arcs and symmetric pairs of arcs are considered up to 
symmetric isotopy and up to swapping an arc $\alpha$ with $\phi(\alpha)$.
\end{definition}

\begin{remark}\label{monogon/digon with double}
In Definition~\ref{sym arc def}, it might seem reasonable to allow $\alpha$ to bound a monogon (or combine with a boundary segment to bound a digon), as long as that monogon (or digon) contains a point in $D$.
However, the symmetry requirement (either $\alpha=\phi(\alpha)$ or $\alpha$ and $\phi(\alpha)$ don't cross) rules out such possibilities.
\end{remark}

\begin{remark}\label{no digon remark}
Definition~\ref{sym arc def} requires that an arc pair $\alpha,\phi(\alpha)$ does not combine to form a digon in $\S^\pm$ because, if so, there is a symmetric arc that should replace them:
Necessarily, that digon contains exactly one non-double fixed point of $\phi$, and there is an arc $\alpha'$ in that digon with $\alpha'=\phi(\alpha')$, having the same endpoints as $\alpha$ and as $\phi(\alpha)$.
On the other hand, if $\alpha$ and $\phi(\alpha)$ combine to form a digon in $\S$ but not in $\S^\pm$, as in the left picture of Figure~\ref{arc digon}, there is no such symmetric arc.
\end{remark}

\begin{definition}[\emph{Ring in a symmetric marked surface with double points}]\label{sym ring def}
A \newword{ring} in $\S^\pm$ is a nontrivial closed curve in $\S$ that has no self-intersections, is disjoint from the boundary of $\S$ and from $\D$, does not, together with an empty boundary component of $\S$, bound an annulus in $\S\setminus\D$, and does not bound a disk in $\S$ only containing an element of $\D$ (i.e.\ one pair of opposite double points in $\D^\pm$).
Furthermore, a ring $U$ must either coincide with $\phi(U)$ or be disjoint from $\phi(U)$, and~$U$ and $\phi(U)$ may not together define an annulus in $\S$ unless that annulus contains a point in $\D$.
The ring $U$ is a \newword{symmetric ring} if $U=\phi(U)$, and otherwise $U,\phi(U)$ is a \newword{symmetric pair of rings}.
Rings are considered up to 
symmetric isotopy.
\end{definition}

\begin{remark}\label{no annulus remark}
Similarly to the definition of arcs (see Remark~\ref{no digon remark}), the definition of rings requires that $U$ and $\phi(U)$ do not combine to bound an empty annulus because such a pair should be replaced by a symmetric ring:
The annulus would contain exactly two non-double fixed points of $\phi$, and there is a ring $U'$ in that annulus, passing through both fixed points, with $U'=\phi(U')$.
\end{remark}

\begin{prop}\label{no ring doub}
No ring in $\S^\pm$ bounds a disk in $\S$ containing only double points.
\end{prop}
\begin{proof}
The definition of a ring rules out rings that bound disks in $\S$ containing~$0$ or~$1$ points in $\D$.
But it is easy to see that a closed curve $U$ that bounds a disk in~$\S$ and has either $\phi(U)=U$ or $\phi(U)$ disjoint from $U$ can contain at most one fixed point under $\phi$.
We omit the easy details.
\end{proof}

\begin{definition}[\emph{Embedded block in a symmetric marked surface with double points}]\label{sym embedded def}
An \newword{embedded block} in $\S^\pm$ is a closed subset $E$ of $\S^\pm$ such that $E$ has nonempty intersection with $\M=\B\cup\D^\pm$, with either $E\cap\phi(E)=\emptyset$ or $E=\phi(E)$, satisfying the following conditions.
\begin{itemize}
\item
If $E\cap\phi(E)=\emptyset$, then $\set{E,\phi(E)}$ is called a \newword{symmetric pair of blocks}. 
We require that
\begin{itemize}
\item $(E,E\cap\M)$ is a (degenerate or nondegenerate) marked surface.
\item
If $(E,E\cap\M)$ is degenerate, it is either a point in $\M=\B\cup\D^\pm$ or an arc or boundary segment of $\S^\pm$ (necessarily with two distinct endpoints).
\item If $(E,E\cap\M)$ is nondegenerate, then each boundary component of $(E,E\cap\M)$ is a ring in $\S^\pm$, an empty boundary component of $\S^\pm$, or a finite union of arcs and/or boundary segments of $\S^\pm$.
\end{itemize}
Note that since $E\cap\phi(E)=\emptyset$ and $\phi$ sends each double point to its opposite double point, $E\cap\M$ contains at most one double point at each location.
\item
If $E=\phi(E)$, then $E$ is called a \newword{symmetric block}. 
We require that
\begin{itemize}
\item $(E,E\cap\M)$ is a degenerate marked surface or, writing $\phi|_E$ for the restriction of $\phi$ to $E$, $(E,(\partial E)\cap\M,\int(E)\cap\D^\pm,\phi|_E)$ is a symmetric marked surface with double points.
\item If $(E,E\cap\M)$ is a degenerate marked surface, then it is a symmetric arc of $\S^\pm$ and is called a \newword{degenerate symmetric block} (possibly with its two endpoints at double points in the same location).
\item
If $(E,(\partial E)\cap\M,\int(E)\cap\D^\pm,\phi|_E)$ is a symmetric marked surface with double points, then each of its boundary components is a ring in $\S^\pm$, an empty boundary component of $\S^\pm$, or a finite union of arcs and/or boundary segments of $\S^\pm$.
We \emph{allow} the boundary of $E$ to contain a double point $d$ and its opposite double point, at different points on the boundary (in the same boundary component or not).
See Remark~\ref{double topology}.
The bottom-left picture in Figure~\ref{sym nc ex fig} show an example.
\end{itemize}
\item
The rings of $(\S^\pm,\B,\D^\pm,\phi)$ occurring as boundary components of $E$ are distinct up to 
symmetric isotopy 
in $(\S^\pm,\B,\D^\pm,\phi)$.
See Remark~\ref{when are rings same}.
\end{itemize}
Embedded blocks are considered up to 
symmetric isotopy.
\end{definition}

\begin{remark}\label{no sym boundary}
A symmetric arc or symmetric ring contains a fixed point of $\phi$ and, locally about that fixed point, $\phi$ is a half-turn rotation.  
For that reason, no symmetric arc or ring can form part of the boundary of $E$, except in the case where $E$ itself is a symmetric arc.
\end{remark}

\begin{remark}\label{double topology}
In the case where $(E,(\partial E)\cap\M,\int(E)\cap\D^\pm,\phi|_E)$ is a symmetric marked surface with double points and $\partial E$ contains a double point $d$ and its opposite, then since every open set in $S^\pm$ containing $d$ also contains the opposite double point, the subspace topology on $E$ is not a surface with double points.  
However, the subspace topology on the interior of $E$ is correct.
\end{remark}

\begin{definition}[\emph{Noncrossing partition of a symmetric surface with double points}]\label{sym nc def}
A \newword{(symmetric) noncrossing partition} of $(\S^\pm,\B,\D^\pm,\phi)$ is a collection $\P$ of \emph{disjoint} embedded blocks, such that the action of $\phi$ permutes the blocks of $\P$, every point in $\M$ is contained in some block of $\P$, and no two distinct blocks have the same ring 
(up to symmetric isotopy) in their boundary.
Noncrossing partitions are considered up to symmetric isotopy.
A specific symmetric-isotopy representative of a noncrossing partition~$\P$ is called an \newword{embedding} of $\P$.
\end{definition}

Recall that Figure~\ref{sym nc ex fig} shows examples of noncrossing partitions of a symmetric surface with double points.

\begin{remark}\label{when are rings same}
Definitions~\ref{sym embedded def} and~\ref{sym nc def} together require that no ring of $\S^\pm$ occurs as the boundary component of two embedded blocks or occurs twice as boundary components of the same block in $\P$. 
A symmetric isotopy 
between rings cannot pass the rings through a fixed point of $\phi$, so two rings are \emph{not} the same up to 
symmetric isotopy 
unless together they bound an annulus in $\S$ containing no fixed points of~$\phi$.
Thus if the two rings together bound an annulus containing one or more double points, then they can both be boundary components of blocks of $\P$ (or of the same block of $\P$).
However, if two rings together bound an annulus in $\S$ containing only non-double fixed points, then they cannot both be boundary components of blocks, because they are excluded from the definition of rings (Definition~\ref{sym ring def}).
\end{remark}

\begin{remark}\label{same location}
Two embedded blocks can be disjoint (and thus can participate together in a noncrossing partition) even when each contains a double point at the same location in~$\S$.
The two blocks necessarily form a symmetric pair.
Examples are found in the top-left and bottom-right pictures in Figure~\ref{sym nc ex fig}.
\end{remark}

\begin{definition}[\emph{Noncrossing partition poset, symmetric-with-double-points case}]\label{sym nc le def}
We set $\P\le\Q$ in $(\S^\pm,\B,\D^\pm,\phi)$ if and only if there exist embeddings of $\P$ and~$\Q$ such that every block of $\P$ is contained in some block of~$\Q$.
The \newword{noncrossing partition poset} is the set $\NC_{(\S^\pm,\B,\D^\pm,\phi)}$ of noncrossing partitions with this partial order.
The symbol $\covered$ denotes cover relations in $\NC_{(\S^\pm,\B,\D^\pm,\phi)}$.
\end{definition}

\begin{example}\label{non lat}
The noncrossing partition poset $\NC_{(\S^\pm,\B,\D^\pm,\phi)}$ can fail to be a lattice.
For example, consider the case where $\S$ is an annulus, $|\D|=2$ and $\B=\emptyset$, with an involutive symmetry $\phi$ that fixes each point in $\D$ and exchanges the two components of the boundary of $\S$.
The noncrossing partition poset $\NC_{(\S^\pm,\B,\D^\pm,\phi)}$ for this case is shown in Figure~\ref{non lat fig}.
\end{example}
\begin{figure}
\scalebox{0.9}{\includegraphics{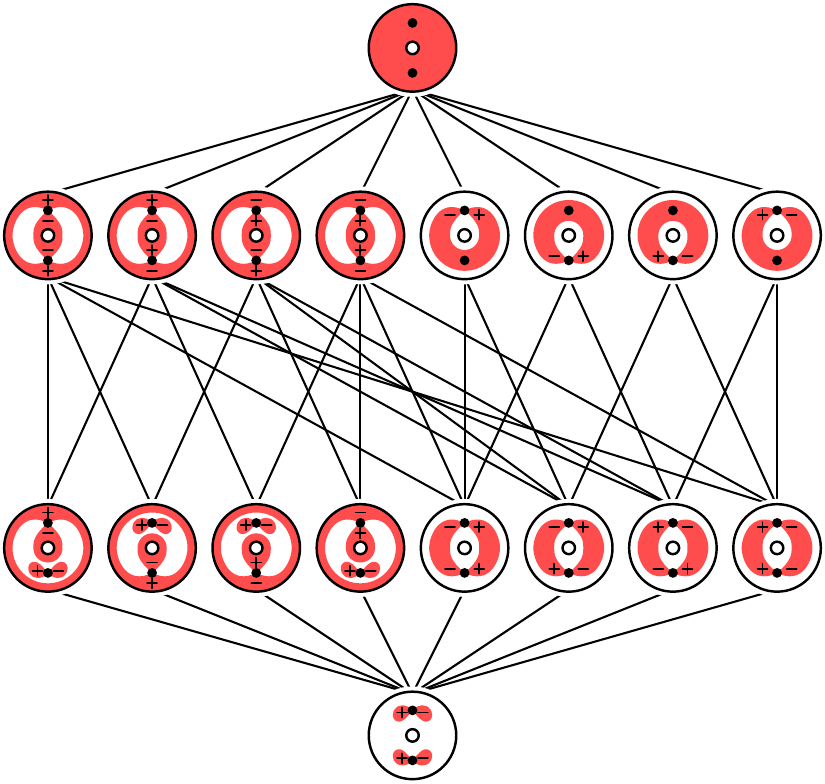}}
\caption{An example where $\NC_{(\S^\pm,\B,\D^\pm,\phi)}$ is not a lattice}
\label{non lat fig}
\end{figure}


We now characterize the rank function of $\NC_{(\S^\pm,\B,\D^\pm,\phi)}$ using the homology of~$\P$, more specifically a different version of Betti numbers.
For $i=0,1$, let $b^\phi_i$ be the dimension of the kernel of the linear map on $H_i(\P)$ sending an $i$-homology class $C$ to $C+\phi(C)$.
Then $b^\phi_0$ counts symmetric pairs $E,\phi(E)$ of distinct embedded blocks of $\P$, because each block is a $0$-homology class and the kernel is spanned by elements of the form $E-\phi(E)$.
Suppose $C$ is an isotopy class of circles in $\P$ that do not bound a disk in $\P$.
If $\phi(C)$ is equal to $C$, then $C$ is not in the kernel.
If $\phi(C)$ is the opposite of $C$ (i.e.\ $C$ with the opposite orientation), then $C$ is in the kernel.
If $\phi(C)$ is isotopic neither to $C$ nor to the opposite of $C$, then $C-\phi(C)$ is in the kernel.
Thus, informally, $b_1^\phi$ counts non-bounding circles $C$ in $\P$ such that $\phi(C)=-C$ and pairs of non-bounding circles $C,\phi(C)$ that are neither identical to nor opposite each other.

Recall that $\M=\B\cup\D^\pm$, so that $|\M|=|\B|+|\D^\pm|=|\B|+2|\D|$.
We will prove the following expression for the rank function, which is similar to Theorem~\ref{graded}.

\begin{theorem}\label{sym graded}
$\NC_{(\S^\pm,\B,\D^\pm,\phi)}$ is a graded poset, with rank function given by 
\[\rank(\P)=\frac12|\M|+b_1^\phi(\P)-b_0^\phi(\P).\]  
\end{theorem}

Once again, lower intervals in the noncrossing partition poset have a natural product decomposition.
The following proposition is immediate from the definitions.

\begin{prop}\label{sym lower}
If $\P$ is a noncrossing partition of $(\S^\pm,\B,\D^\pm,\phi)$ with nondegenerate symmetric blocks $E_1,\ldots,E_k$, degenerate symmetric blocks $F_1,\ldots,F_\ell$, and symmetric pairs $\set{F_{\ell+1},\phi(F_{\ell+1})},\ldots,\set{F_m,\phi(F_m)}$ of blocks, then the interval below $\P$ in $\NC_{(\S^\pm,\B,\D^\pm,\phi)}$ is isomorphic to 
\[\prod_{i=1}^k\NC_{(E_i,\partial(E_i)\cap\M,\int(E_i)\cap\D^\pm,\phi|_{E_i})}\times\prod_{i=1}^m\NC_{(F_i,F_i\cap\M)}.\]
\end{prop}

The proof of Theorem~\ref{sym graded} follows the same general outline as the analogous proofs without symmetry/double points. 
The fact that $\NC_{(\S^\pm,\B,\D^\pm,\phi)}$ is a bounded poset follows by the same proof as Lemma~\ref{poset}.
To continue the proof, we again characterize noncrossing partitions in terms of certain curves they contain.

\begin{definition}[\emph{Curve set, symmetric-with-double-points case}]\label{sym curve set def}
The \newword{curve set} $\curve(E)$ of an embedded block~$E$ is the set of all arcs and boundary segments of $\S^\pm$ that 
(up to symmetric isotopy) 
are contained in~$E$.
The \newword{curve set} $\curve(\P)$ of a noncrossing partition $\P$ is the union of the curve sets of its blocks.
The curve set is closed under the action of $\phi$.
\end{definition}

The following propositions hold by the same proofs as for marked surfaces (Propositions~\ref{curve set le} and~\ref{curve set det}).

\begin{prop}\label{sym curve set le}
Two symmetric noncrossing partitions $\P$ and $\Q$ of $\S^\pm$ have $\P\le\Q$ if and only if $\curve(\P)\subseteq\curve(\Q)$.
\end{prop}

\begin{prop}\label{sym curve set det}
A noncrossing partition of $(\S^\pm,\B,\D^\pm,\phi)$ is uniquely determined 
(up to symmetric isotopy)
 by its curve set.
\end{prop}


To describe cover relations in $\NC_{(\S^\pm,\B,\D^\pm,\phi)}$, we adapt Definition~\ref{simp conn} to accommodate double points and symmetry.  

\begin{definition}[\emph{Simple symmetric connector/pair and augmentation}]\label{sym simp conn}
Suppose~$\P$ is a noncrossing partition of $(\S^\pm,\B,\D^\pm,\phi)$.
A \newword{simple symmetric connector} is a symmetric arc $\alpha$ that is a simple connector for $\P$ 
(taking the definition of simple connector from Definition~\ref{simp conn} verbatim, with symmetric isotopy replacing ordinary isotopy), 
but ruling out one case:
\begin{itemize}
\item
$\alpha$ may not combine with blocks of $\P$ to bound a disk in $\S$ containing a pair of opposite double points in~$\D^\pm$ that are trivial blocks in $\P$.
See the top-left picture of Figure~\ref{simpexclude fig}.
(The picture shows an annulus with one double point and one non-double $\phi$-fixed point.)
\end{itemize}
\begin{figure}
\scalebox{0.5}{\includegraphics{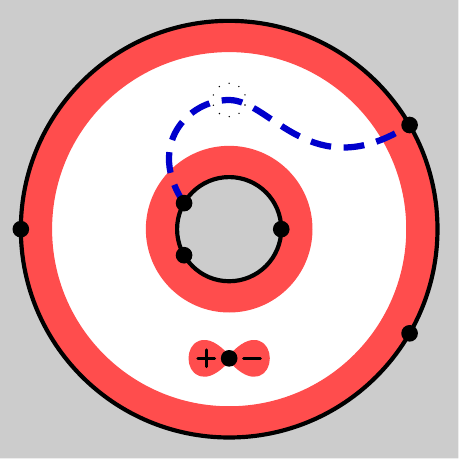}}
\quad
\scalebox{0.5}{\includegraphics{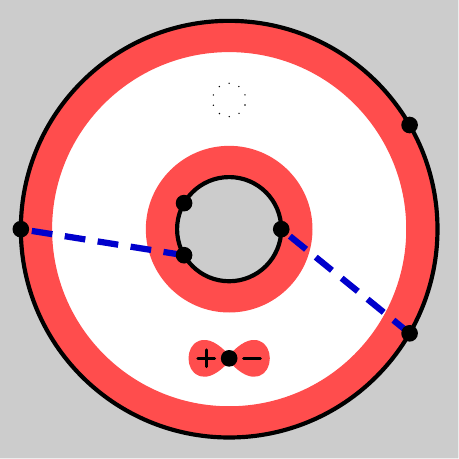}}
\quad
\scalebox{0.5}{\includegraphics{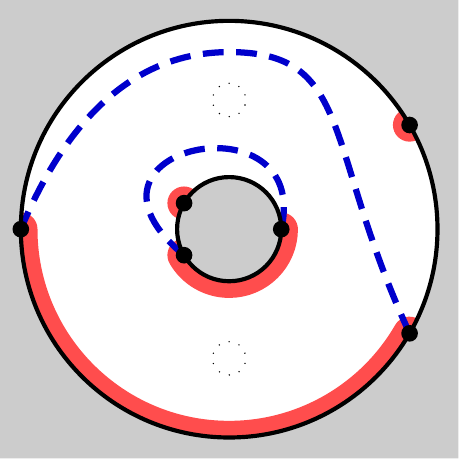}}\\[12pt]
\scalebox{0.5}{\includegraphics{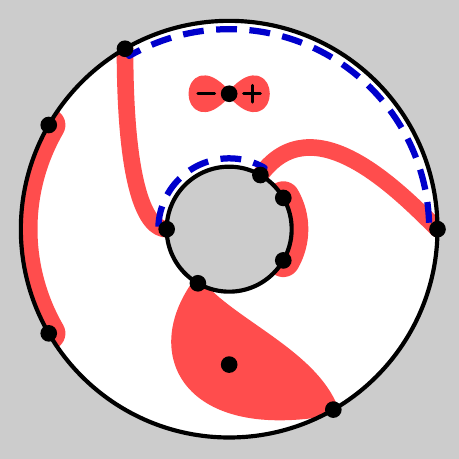}}
\quad
\scalebox{0.5}{\includegraphics{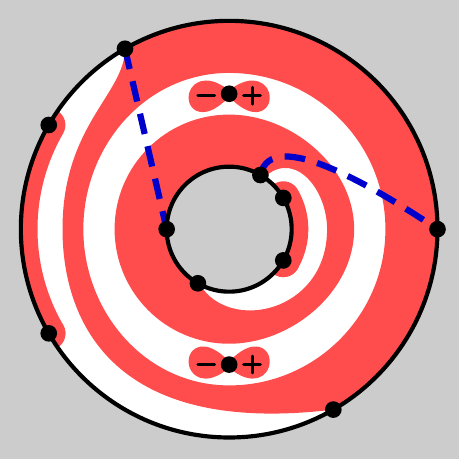}}
\quad
\scalebox{0.5}{\includegraphics{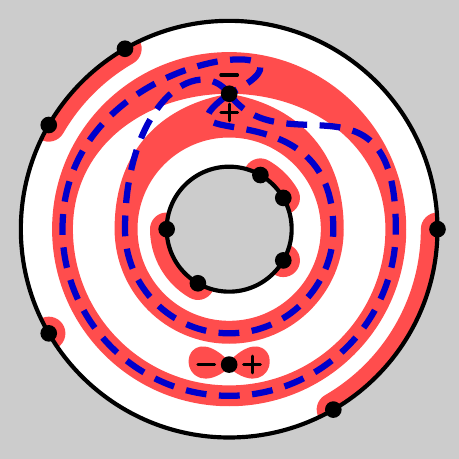}}
\caption{Some excluded cases in the definition of simple symmetric connectors/pairs}
\label{simpexclude fig}
\end{figure}
The \newword{augmentation of $\P$ along $\alpha$} is defined exactly as in Definition~\ref{simp conn}, with the thickenings in the curve union chosen symmetrically.
This augmentation is denoted $\P\cup\alpha\cup\phi(\alpha)$, since $\alpha=\phi(\alpha)$.

A \newword{simple symmetric pair of connectors} is a symmetric pair $\alpha,\phi(\alpha)$ of arcs or boundary segments, each of which is a simple connector for $\P$, but ruling out two possibilities:
\begin{itemize}
\item
$\alpha$ and $\phi(\alpha)$ may not combine with blocks of $\P$ to bound a disk in $\S$ containing a pair of opposite double points in~$\D^\pm$ that are trivial blocks in $\P$.
See the top-middle picture and bottom pictures of Figure~\ref{simpexclude fig}.
(The top-middle picture again shows an annulus with one double point and one non-double $\phi$-fixed point, while the bottom picture shows an annulus with two double points.)
\item
$\alpha$ and $\phi(\alpha)$ may not combine with blocks of $\P$ to bound an annulus containing only non-double fixed points of $\phi$. 
(In this excluded case, the annulus would contain exactly $2$ fixed points.)
See the top-right picture of Figure~\ref{simpexclude fig}.
(The picture shows an annulus with no double points and two non-double $\phi$-fixed points.)
\end{itemize}
We emphasize the requirement that $\alpha,\phi(\alpha)$ is a symmetric pair of arcs or boundary segments.
In particular, by definition for arcs and obviously for boundary segments, $\alpha$ and $\phi(\alpha)$ are disjoint except possibly at endpoints.

The \newword{augmentation of $\P$ along $\alpha$ and $\phi(\alpha)$}, denoted $\P\cup\alpha\cup\phi(\alpha)$, is obtained as the augmentation along $\alpha$ as in Definition~\ref{simp conn}, further augmented along $\phi(\alpha)$, with the thickenings in the curve unions chosen so as to make $\P\cup\alpha\cup\phi(\alpha)$ symmetric, and in two cases, appending to additional disks to the curve union.
Specifically, if $\alpha$ and $\phi(\alpha)$ combine with blocks of $\P$ to bound a disk containing no double points (but necessarily containing a non-double fixed point of $\phi$), then adjoin that disk to the curve union, as illustrated in the top three pictures in Figure~\ref{diskfill fig}.
Also, it is possible that $\alpha$ and $\phi(\alpha)$, together with curves in $E$ and $E'$ (not excluding the case where $E=E'$), bound a disk containing a pair $d_+,d_-$ of double points, with $d_+$ in $E$ and $d_-$ in $E'$ or vice versa.
In this case, the union of $E$ and $E'$ with the thickened $\alpha$ and $\phi(\alpha)$ has some empty disks that must be filled in, as illustrated in the bottom three pictures in Figure~\ref{diskfill fig}.
\end{definition}

\begin{figure}
\scalebox{0.5}{\includegraphics{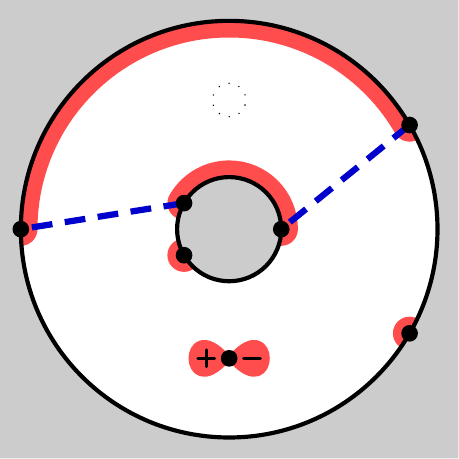}}
\quad
\scalebox{0.5}{\includegraphics{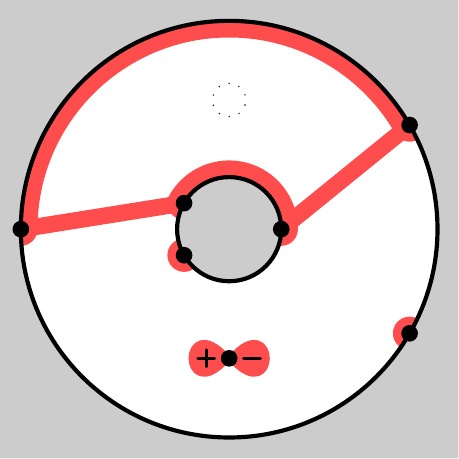}}
\quad
\scalebox{0.5}{\includegraphics{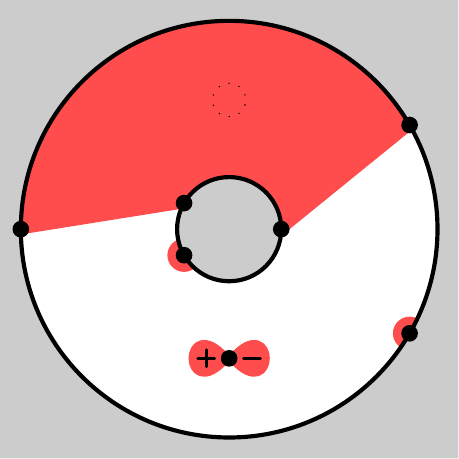}}\\[12pt]
\scalebox{0.5}{\includegraphics{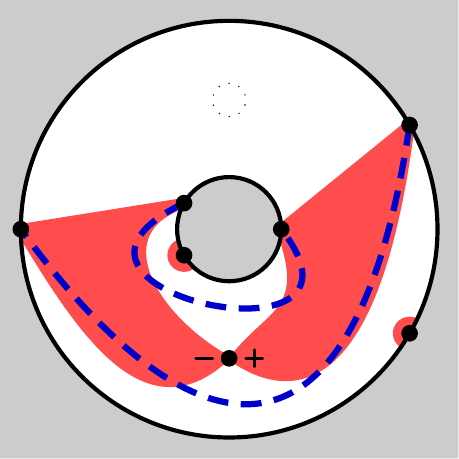}}
\quad
\scalebox{0.5}{\includegraphics{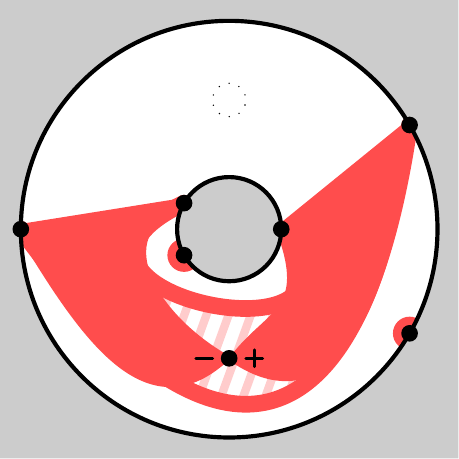}}
\quad
\scalebox{0.5}{\includegraphics{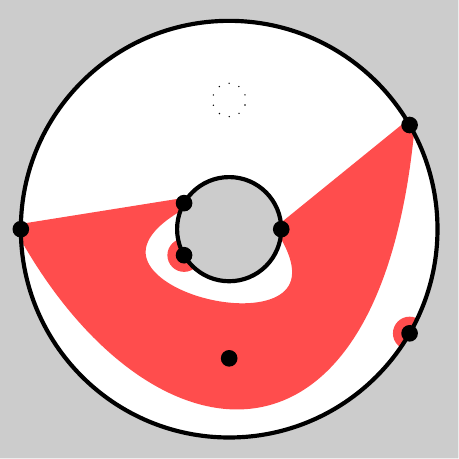}}
\caption{Filling in empty disks in the construction of the augmentation}
\label{diskfill fig}
\end{figure}

\begin{remark}\label{sym aug rem}
In Definition~\ref{sym simp conn}, ruling out the case where $\alpha$ combines (or $\alpha$ and $\phi(\alpha)$ combine) with blocks of $\P$ to bound a disk in $\S$ containing points in $\D$ is necessary to ensure that all components of the boundary of $\P\cup\alpha\cup\phi(\alpha)$ that are disjoint from the boundary of $\S$ are indeed rings in $\S^\pm$.
(See Proposition~\ref{no ring doub}.)
Ruling out the case where $\alpha$ combines (or $\alpha$ and $\phi(\alpha)$ combine) with blocks of~$\P$ to bound an annulus containing only non-double fixed points of $\phi$ is necessary for a similar reason.
Without this requirement, the augmentation would have a boundary component $U$ that fails to be a ring because it combines with $\phi(U)$ to form an annulus not containing any double points.
%
We do not need to rule out the case where a simple symmetric connector $\alpha$ combines with blocks of $P$ to bound an annulus containing only non-double fixed points of $\phi$, because in that case, the boundary of the annulus would contain a point (in $\alpha$) fixed by $\phi$, which is impossible.
\end{remark}

\begin{figure}
\scalebox{0.37}{\includegraphics{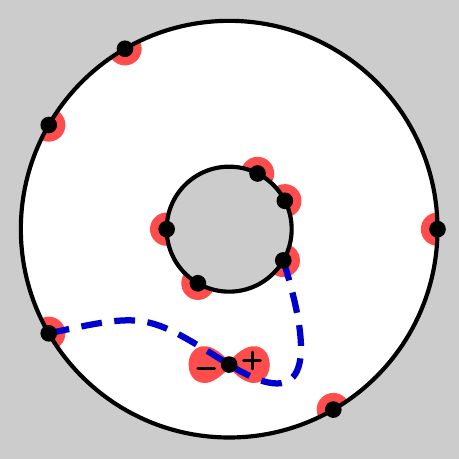}}
\,
\scalebox{0.37}{\includegraphics{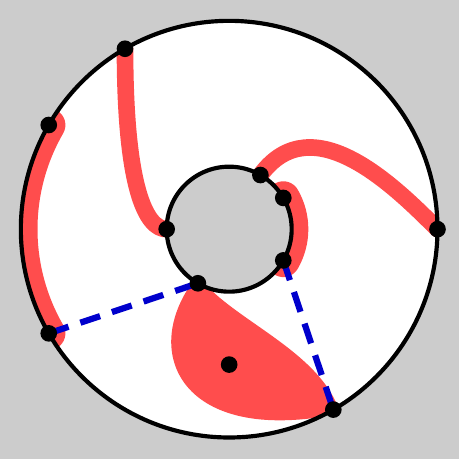}}
\,
\scalebox{0.37}{\includegraphics{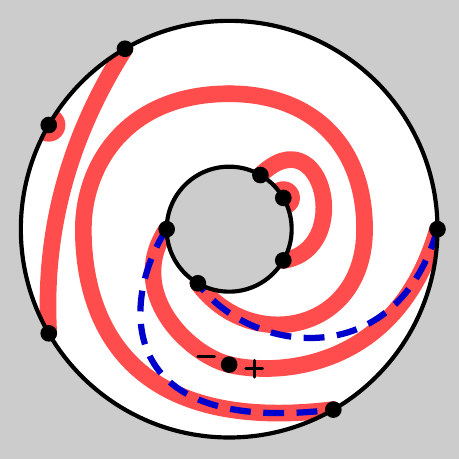}}
\,
\scalebox{0.37}{\includegraphics{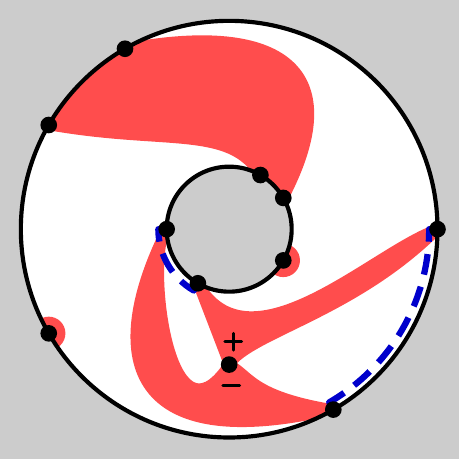}}\\[5pt]
\scalebox{0.37}{\includegraphics{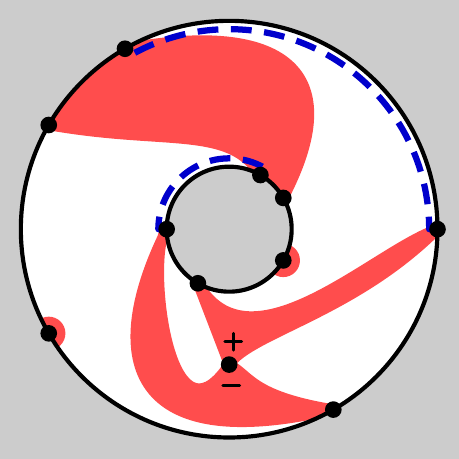}}
\,
\scalebox{0.37}{\includegraphics{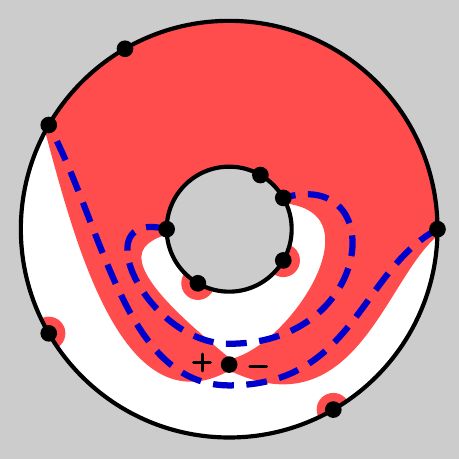}}
\,
\scalebox{0.37}{\includegraphics{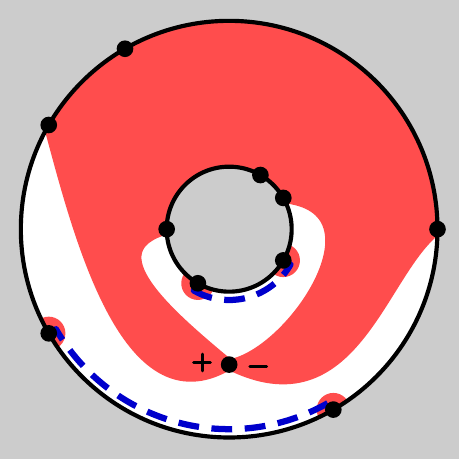}}
\,
\scalebox{0.37}{\includegraphics{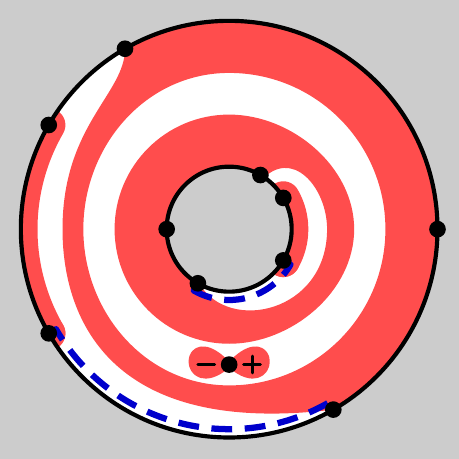}}\\[5pt]
\scalebox{0.37}{\includegraphics{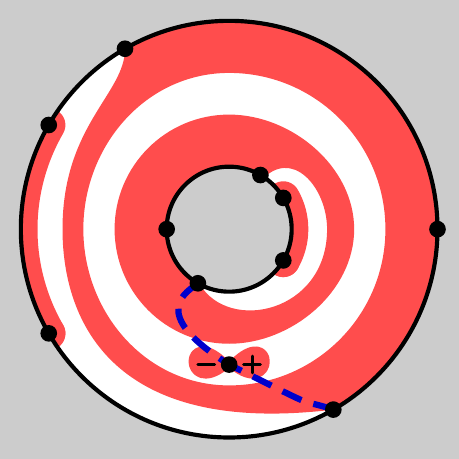}}
\,
\scalebox{0.37}{\includegraphics{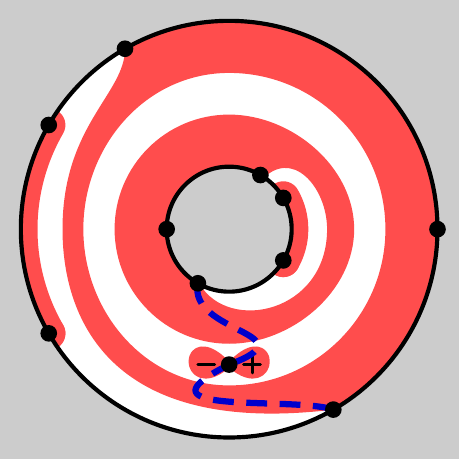}}
\,
\scalebox{0.37}{\includegraphics{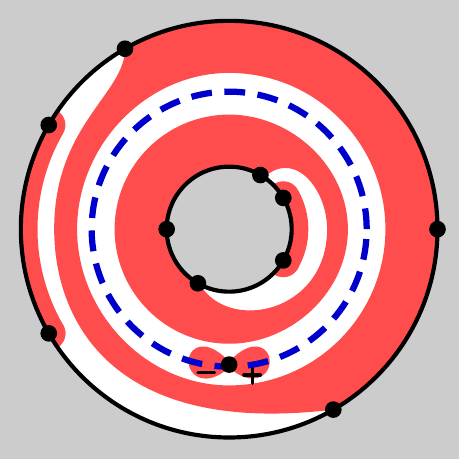}}
\,
\scalebox{0.37}{\includegraphics{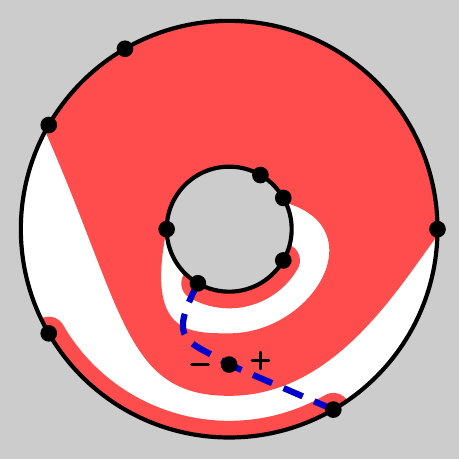}}\\[5pt]
\scalebox{0.37}{\includegraphics{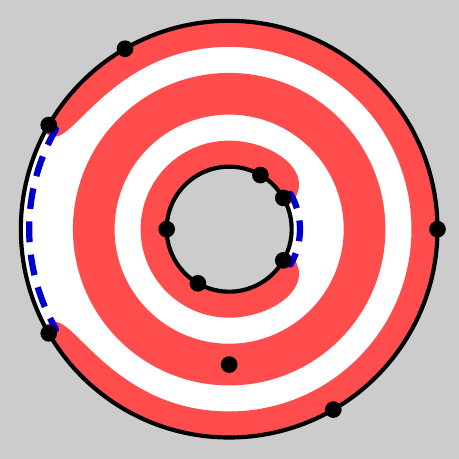}}
\,
\scalebox{0.37}{\includegraphics{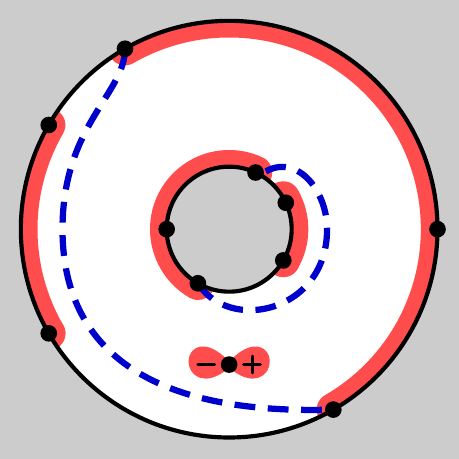}}
\,
\scalebox{0.37}{\includegraphics{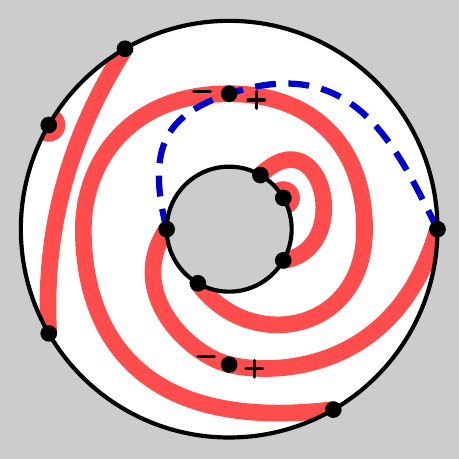}}
\,
\scalebox{0.37}{\includegraphics{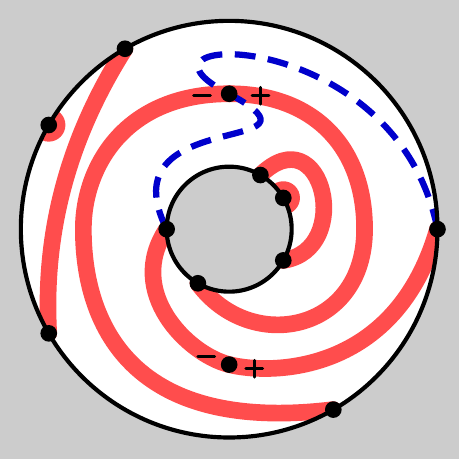}}\\[5pt]
\scalebox{0.37}{\includegraphics{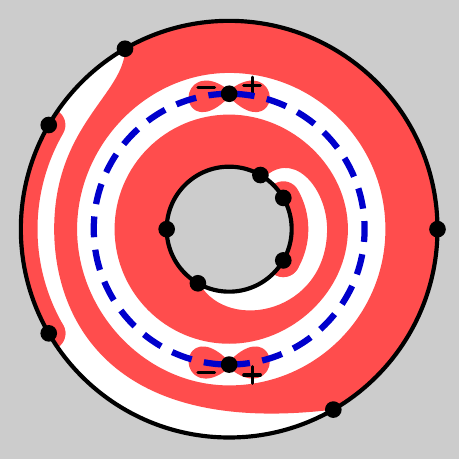}}
\,
\scalebox{0.37}{\includegraphics{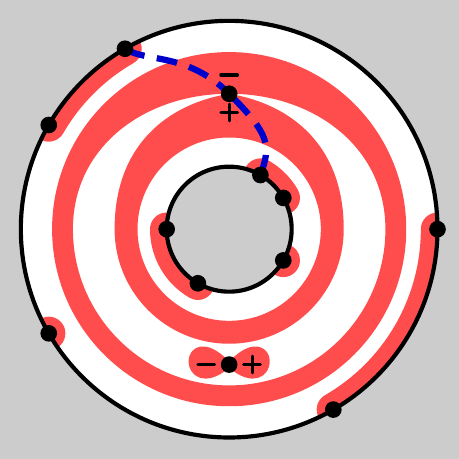}}
\,
\scalebox{0.37}{\includegraphics{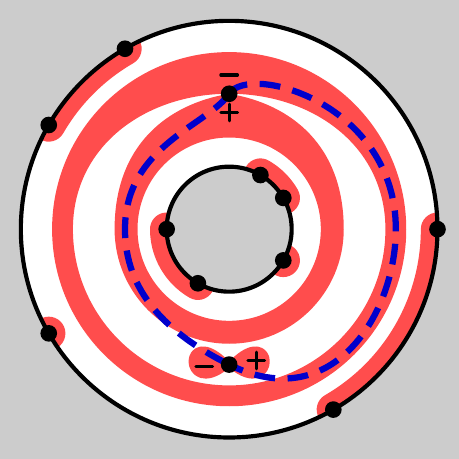}}
\,
\scalebox{0.37}{\includegraphics{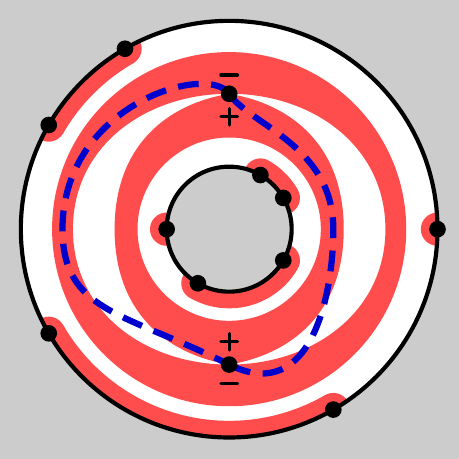}}\\[5pt]
\scalebox{0.37}{\includegraphics{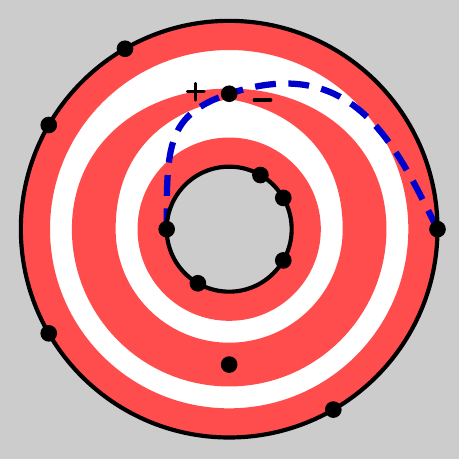}}
\,
\scalebox{0.37}{\includegraphics{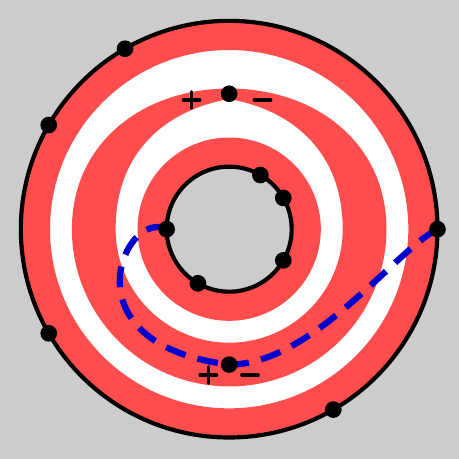}}
\,
\scalebox{0.37}{\includegraphics{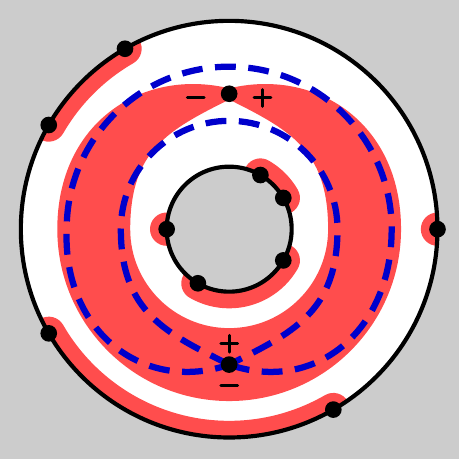}}
\,
\scalebox{0.37}{\includegraphics{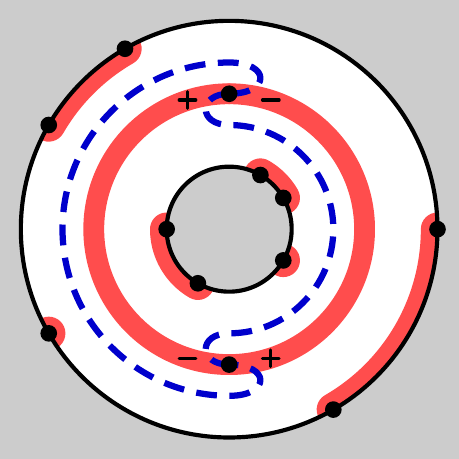}}
\caption{Simple symmetric connectors/pairs of connectors in the symmetric disk with one or two double points}
\label{simp pair}
\end{figure}

\begin{figure}
\scalebox{0.37}{\includegraphics{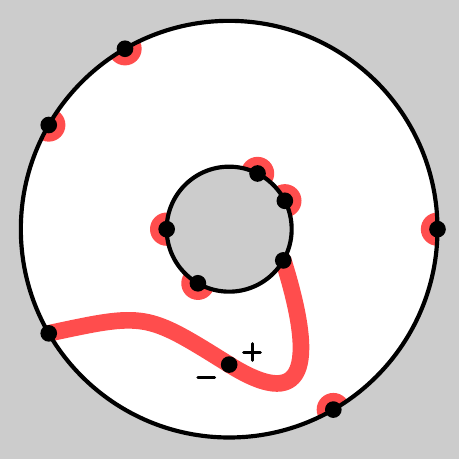}}
\,
\scalebox{0.37}{\includegraphics{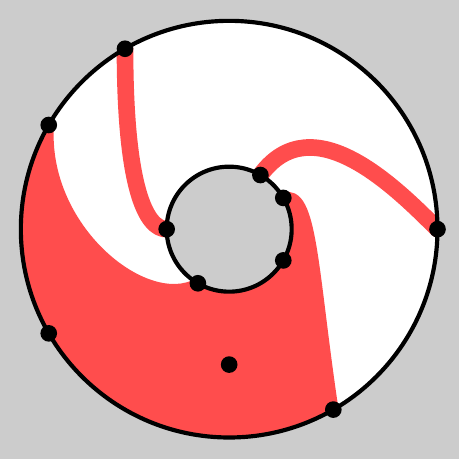}}
\,
\scalebox{0.37}{\includegraphics{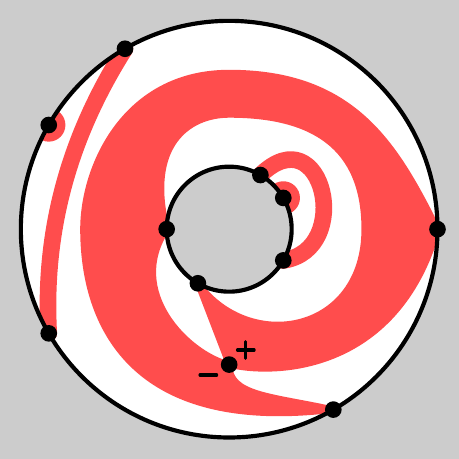}}
\,
\scalebox{0.37}{\includegraphics{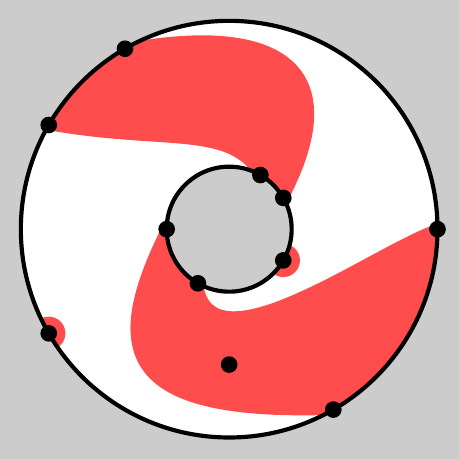}}\\[5pt]
\scalebox{0.37}{\includegraphics{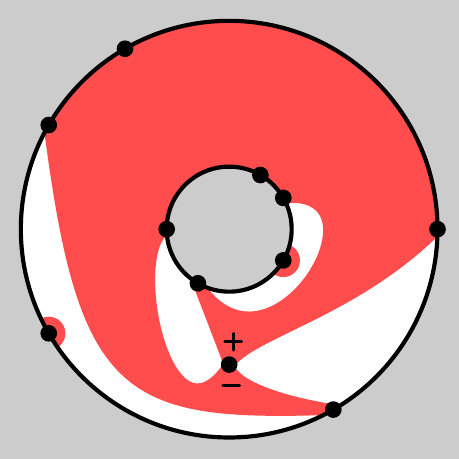}}
\,
\scalebox{0.37}{\includegraphics{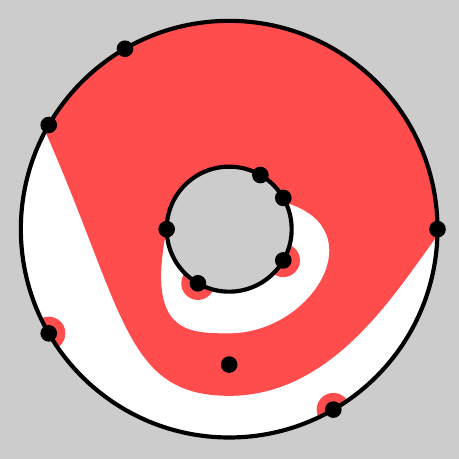}}
\,
\scalebox{0.37}{\includegraphics{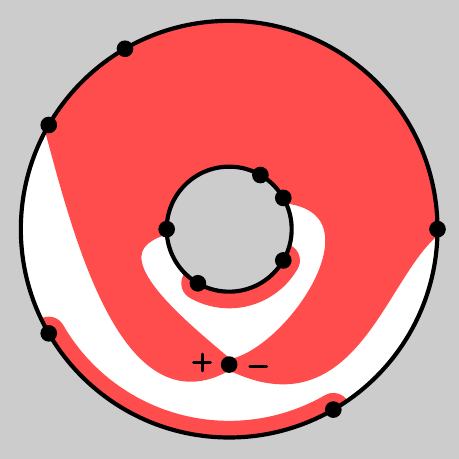}}
\,
\scalebox{0.37}{\includegraphics{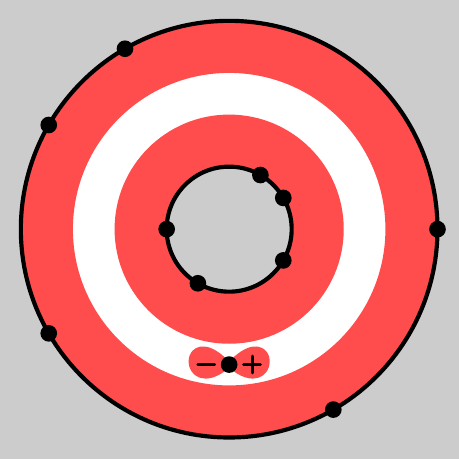}}\\[5pt]
\scalebox{0.37}{\includegraphics{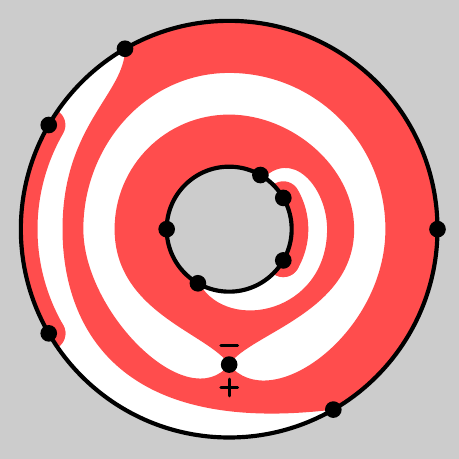}}
\,
\scalebox{0.37}{\includegraphics{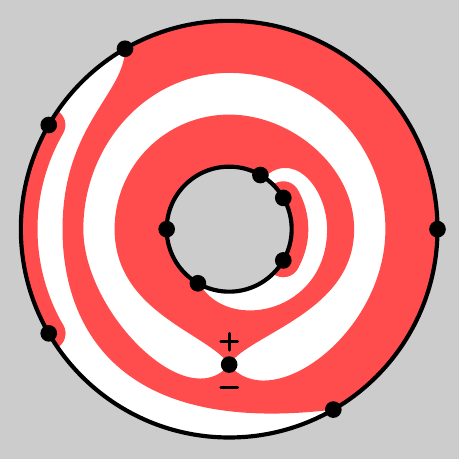}}
\,
\scalebox{0.37}{\includegraphics{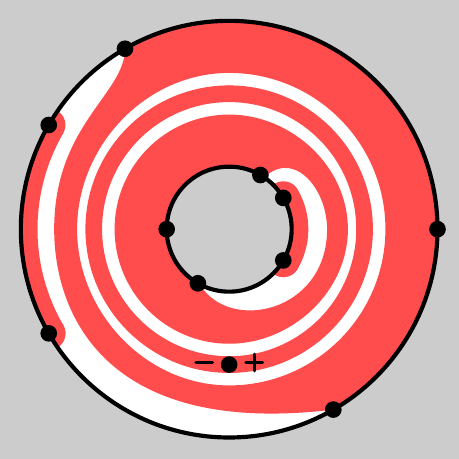}}
\,
\scalebox{0.37}{\includegraphics{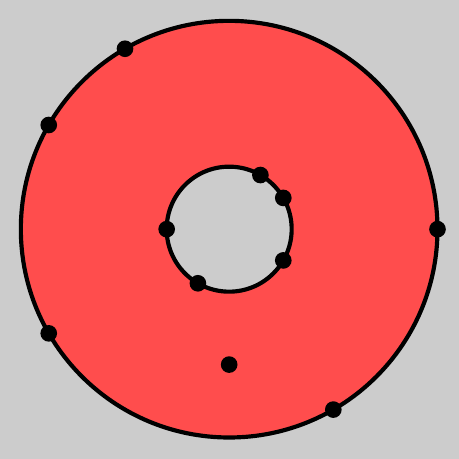}}\\[5pt]
\scalebox{0.37}{\includegraphics{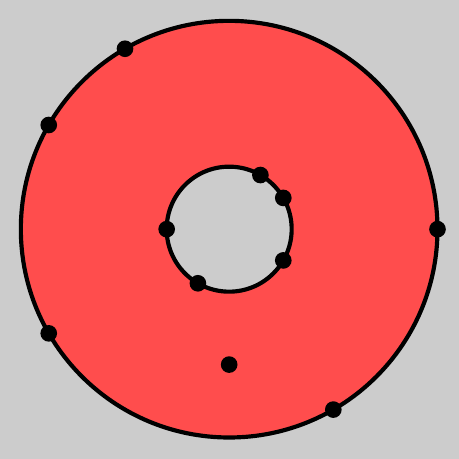}}
\,
\scalebox{0.37}{\includegraphics{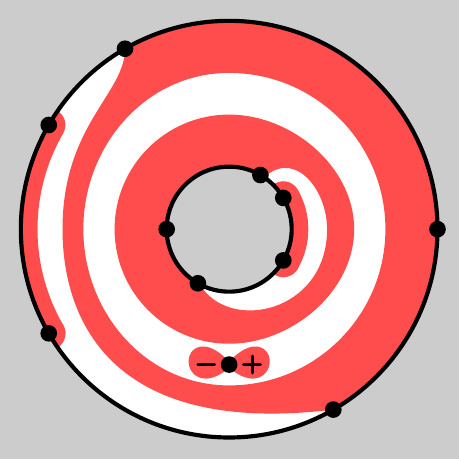}}
\,
\scalebox{0.37}{\includegraphics{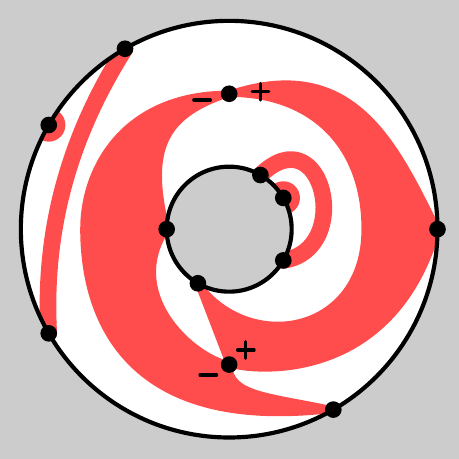}}
\,
\scalebox{0.37}{\includegraphics{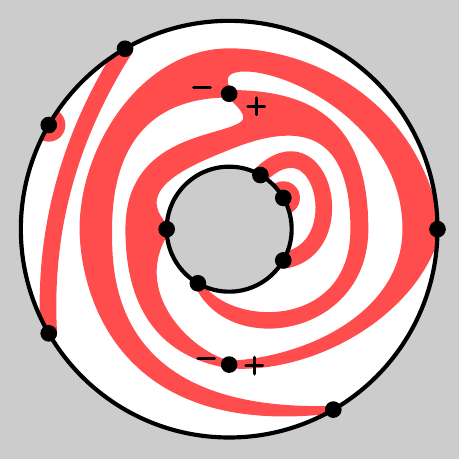}}\\[5pt]
\scalebox{0.37}{\includegraphics{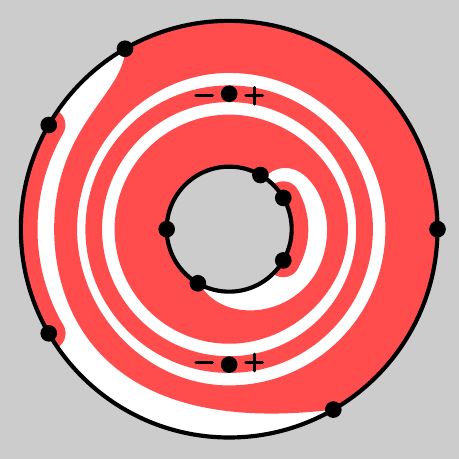}}
\,
\scalebox{0.37}{\includegraphics{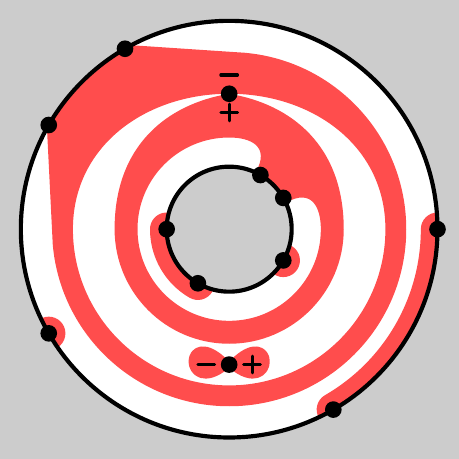}}
\,
\scalebox{0.37}{\includegraphics{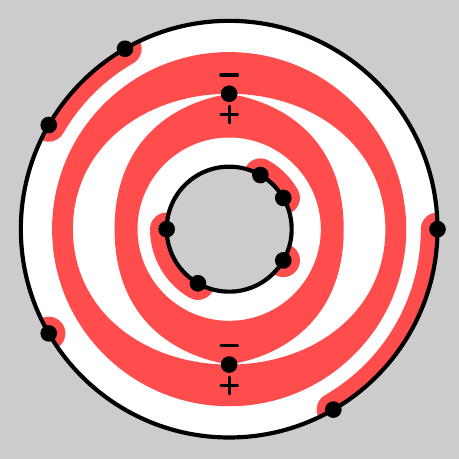}}
\,
\scalebox{0.37}{\includegraphics{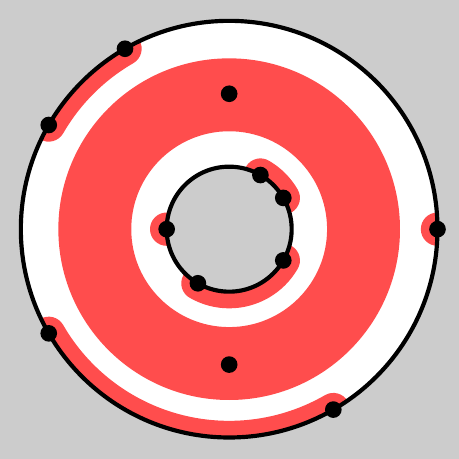}}\\[5pt]
\scalebox{0.37}{\includegraphics{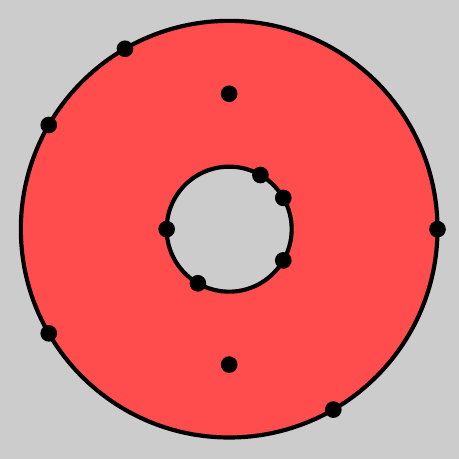}}
\,
\scalebox{0.37}{\includegraphics{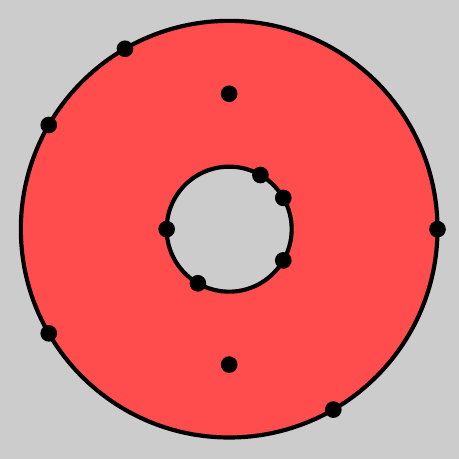}}
\,
\scalebox{0.37}{\includegraphics{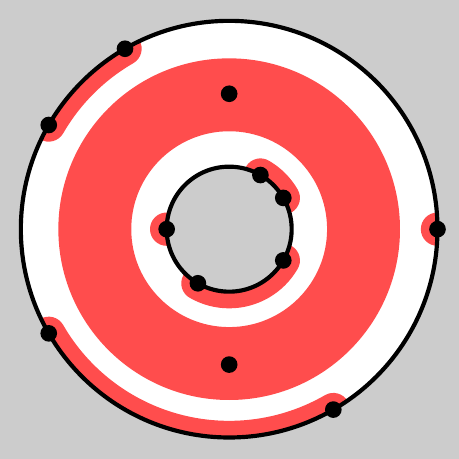}}
\,
\scalebox{0.37}{\includegraphics{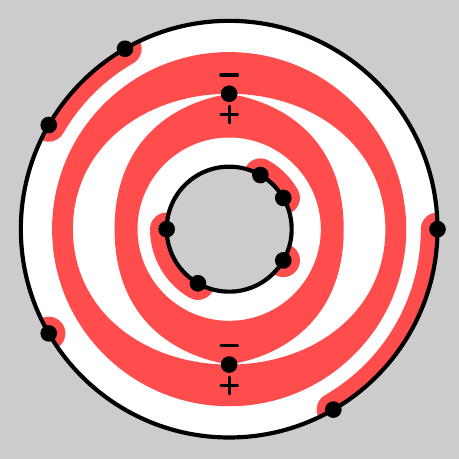}}
\caption{Covers associated to the simple symmetric connectors/pairs of connectors in Figure~\protect\ref{simp pair}}
\label{covs}
\end{figure}

\begin{example}\label{simp cov ex}
Figure~\ref{simp pair} shows some simple symmetric connectors and simple symmetric pairs of connectors.
For each picture in Figure~\ref{simp pair}, the corresponding picture in Figure~\ref{covs} shows the augmentation.  
In light of Proposition~\ref{sym covers}, corresponding pictures in Figures~\ref{simp pair} and~\ref{covs} represent cover relations.
We emphasize that two different simple symmetric pair of connectors for $\P$ can produce the same augmentation of $\P$ and thus the same cover relation.
\end{example}

The following lemma holds by the same proof as Lemma~\ref{alpha does it}.

\begin{lemma}\label{sym alpha does it}
Suppose $\P$ is a noncrossing partition of $(\S^\pm,\B,\D^\pm,\phi)$ and suppose~$\alpha$ is a simple symmetric connector for $\P$ or $\alpha,\phi(\alpha)$ is a simple symmetric pair of connectors for $\P$.
If $\R$ is a noncrossing partition of $(\S^\pm,\B,\D^\pm,\phi)$ with $\P\le\R$ and $\alpha,\phi(\alpha)\in\curve(\R)$, then $\P\cup\alpha\cup\phi(\alpha)\le\R$.
\end{lemma}

The characterization of covers in $\NC_{(\S^\pm,\B,\D^\pm,\phi)}$ follows, and partly reuses, the argument for $\NC_{(\S,\M)}$ (Proposition~\ref{covers}).

\begin{prop}\label{sym covers}
Two noncrossing partitions $\P,\Q\in\NC_{(\S^\pm,\B,\D^\pm,\phi)}$ have $\P\covered\Q$ if and only if there exists a simple symmetric connector $\alpha$ or simple symmetric pair of connectors $\alpha,\phi(\alpha)$ for $\P$ such that $\Q=\P\cup\alpha\cup\phi(\alpha)$.
\end{prop}
\begin{proof}
Suppose $\alpha$ is a simple symmetric connector or $\alpha,\phi(\alpha)$ is a simple symmetric pair of connectors such that $\Q=\P\cup\alpha\cup\phi(\alpha)$.
Then $\P<\Q$, and we will show that $\P\covered\Q$ by showing that any noncrossing partition $\R$ with $\P<\R\le\Q$ has $\R=\Q$.
In light of Proposition~\ref{sym alpha does it}, it is enough to show that $\alpha,\phi(\alpha)\in\curve(\R)$, and since $\curve(\R)$ is closed under the action of $\phi$, it is enough to show that $\alpha\in\curve(\R)$ or $\phi(\alpha)\in\curve(\R)$.

Consider a curve $\beta$ in $\curve(\R)\setminus\curve(\P)$.
If $\beta$ follows along the part of $\alpha$ not in a block of $\P$, then certainly $\alpha\in\curve(\R)$, and if $\beta$ follows $\phi(\alpha)$, then $\phi(\alpha)\in\curve(\R)$.
If $\beta$ follows neither $\alpha$ nor $\phi(\alpha)$, then we argue as in the proof of Proposition~\ref{covers} (using Proposition~\ref{sym curve set det} in place of Proposition~\ref{curve set det}) that $\alpha$ or $\phi(\alpha)$ is in $\curve(\R)$.
We can use this argument because the exclusions in Definition~\ref{sym simp conn} ensure that, in the construction of the augmentation of $\P$ along $\alpha$ and $\phi(\alpha)$, the (thickened) unions of blocks and curves are bounded by \emph{rings}, and then blocks are further combined along common rings.

Conversely, if $\P\covered\Q$, then Proposition~\ref{sym curve set det} says in particular that there exists a curve $\beta\in\curve(\Q)\setminus\curve(\P)$.
Arguing as in the proof of Proposition~\ref{covers}, we use~$\beta$ to construct an arc or boundary segment $\alpha\in\curve(\Q)$ that has no 
symmetric-isotopy
representative inside a single block of $\P$ but has 
a symmetric-isotopy 
representative that starts inside a block $E$ of $\P$.
The curve(s) $\alpha$ (and~$\phi(\alpha)$) satisfy all the requirements to be a simple symmetric connector (or simple symmetric pair of connectors) as long as they are not explicitly ruled out in Definition~\ref{sym simp conn}.
If $\alpha$ (and~$\phi(\alpha)$) are not ruled out in Definition~\ref{sym simp conn}, then  by the direction of this proposition already proved and by Lemma~\ref{sym alpha does it}, $\P\covered\P\cup\alpha\cup\phi(\alpha)\le\Q$, so $\Q=\P\cup\alpha\cup\phi(\alpha)$.

In the cases where $\alpha$ (and $\phi(\alpha)$) are ruled out in Definition~\ref{sym simp conn}, we will construct a simple connector $\gamma$ for $\P$ with $\gamma\in\curve(\Q)$, and conclude similarly that $\Q=\P\cup\gamma\cup\phi(\gamma)$.
(In fact, if we wished to argue further, we could show that $\alpha\not\in\curve(P\cup\gamma\cup\phi(\gamma))$, and by that contradiction conclude that $\alpha,\phi(\alpha)$ is not ruled out in Definition~\ref{sym simp conn}.)
It may be helpful to refer to Figure~\ref{simpexclude fig}.

First, suppose $\alpha$ connects blocks $E$ and $E'$ of $\P$ and suppose that $\alpha$ and $\phi(\alpha)$ combine with blocks of $\P$ to bound a disk containing a pair of opposite double points.
(We argue together the cases where $\alpha$ is a symmetric arc or a symmetric pair.) 
Necessarily, $E'=\phi(E)$ so that $\phi(\alpha)$ also connects $E$ to $E'$.
The sets $E$, $E'$, $\alpha$ and $\phi(\alpha)$ are contained in the same block $F$ of $\Q$.
If a pair $d_+$ and $d_-$ of double points in the disk is not in $F$, then it is separated from $F$ by a closed curve that bounds a disk containing double points.
But by Proposition~\ref{no ring doub}, such a closed curve is not a ring, so we see that $d_+$ and $d_-$ are in $F$.
There is a simple symmetric pair of connectors $\gamma,\phi(\gamma)$ such that $\gamma$ follows $\alpha$ from its endpoint in $E$ to the boundary of $E$ but then goes to $d_+$ instead of continuing with $\alpha$.
Furthermore, $\gamma\in\curve(\Q)$ as desired.

Next, suppose $\alpha$ connects blocks $E$ and $E'$ and that the union of $E$ and $E'$ and a thickened $\alpha$ has a component in its boundary that bounds (together with its image under $\phi$) an annulus $A$ containing $2$ fixed points of $\phi$ that are not double points.
The other boundary circle of $A$ is a component of the boundary of the union of $\phi(E)$ and $\phi(E')$ and a thickened $\phi(\alpha)$.
Neither boundary of $A$ is a ring, so since $E$, $E'$, $\alpha$, and $\phi(\alpha)$ are all contained in blocks of $\Q$, the annulus $A$ is also contained in $\Q$.
Let $\gamma$ be a symmetric arc starting in $E$, passing through $A$ (necessarily through a fixed point of $\phi$, and then ending in $E$.
In particular, $\gamma\in\curve(\Q)$.
We can show that $\gamma$ is a simple symmetric connector for $\P$ by verifying that $\gamma$ does not combine with blocks of $\P$ to bound a disk containing double points.
But if so, then $\alpha$ (and possibly $\phi(\alpha)$) goes through that disk, breaking it into smaller disks.
Since $A$ contains no double points, $\alpha$ combines with blocks of $\P$ to bound a disk with double points, and we are in the previous case, so we are done.
Otherwise, $\gamma$ is a simple symmetric connector for $\P$, and $\gamma\in\curve(\Q)$.

We have shown in every case that we can augment $\P$ along a simple symmetric connector or a simple symmetric pair of connectors to obtain $\Q$.
\end{proof}

The proof of Theorem~\ref{sym graded} is similar to the proof of Theorem~\ref{graded}, but has more cases.

\begin{proof}[Proof of Theorem~\ref{sym graded}]
The formula for the rank function is correct when $\P$ is the bottom element, and we will verify that going up by a cover relation adds $1$ to $b_1^\phi(\P)-b_0^\phi(\P)$.
The construction of the augmentation first combines blocks of $\P$ along a simple symmetric connector or simple symmetric pair of connectors, and then inserts annuli to fulfill the requirement that the same ring may not appear twice as a component of boundaries of blocks.
We will first show that the initial combination of blocks adds $1$ to $b_1^\phi(\P)-b_0^\phi(\P)$.
We will then show that the subsequent insertion of annuli leaves $b_1^\phi(\P)-b_0^\phi(\P)$ unchanged.

Given a simple symmetric connector $\alpha$ or simple symmetric pair of connectors $\alpha,\phi(\alpha)$, with $\alpha$ connecting $E$ to $E'$, we break into cases defined by whether $\alpha$ is symmetric or a pair and by coincidences among the blocks $E$, $E'$, $\phi(E)$, and $\phi(E')$.
We check all cases up to the symmetry of swapping $E$ and $E'$ and/or swapping $E$ and $\phi(E)$.
In each case, we show that $b_1^\phi(\P)-b_0^\phi(\P)$ increases by $1$ when blocks of~$\P$ are joined using $\alpha$ and $\phi(\alpha)$.

First, suppose $\alpha$ is a symmetric simple connector.  
These cases are illustrated in Figure~\ref{symrank fig}.
(In this figure and in Figure~\ref{symrank fig 2}, the surface is an annulus with one double point and one non-double fixed point, except as specified in Cases~7 and ~8.)
In this case, $E'=\phi(E)$.

\noindent
\textbf{Case 1.}
If $\phi(E)\neq E$, then joining $E$ and $E'$ along $\alpha$ decreases $b_0^\phi$ by $1$ while leaving $b_1^\phi$ unchanged. 

\noindent
\textbf{Case 2.}
If $\phi(E)=E$, then joining $E$ and $E'$ along $\alpha$ leaves $b_0^\phi$ unchanged, while increasing $b_1^\phi$ by $1$.
(A new non-bounding circle is obtained by taking a path in $E$ connecting the points where $\alpha$ enters and leaves $E$ and concatenating it with the part of $\alpha$ that is outside of $E$.
Since the part of $\alpha$ outside of $E$ contains a fixed point of $\phi$, the action of $\phi$ either reverses the orientation of this circle or sends it to a distinct circle, depending on the path chosen in~$E$.)

\begin{figure}
\begin{tabular}{cc}
\scalebox{0.5}{\includegraphics{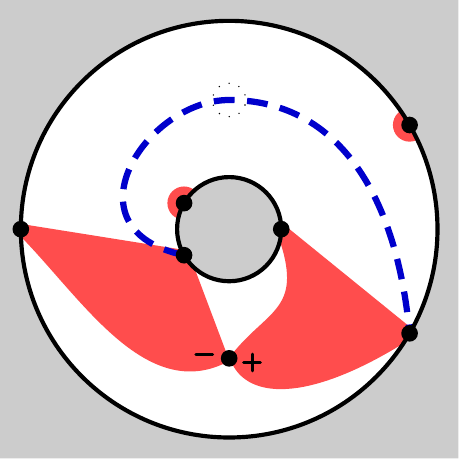}}&
\scalebox{0.5}{\includegraphics{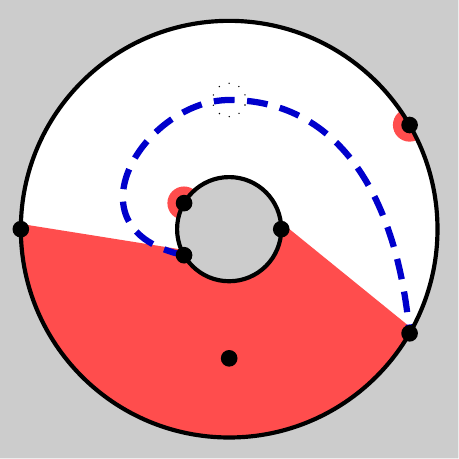}}\\
Case 1&Case 2
\end{tabular}
\caption{Illustrations for the proof of Theorem~\ref{sym graded}}
\label{symrank fig}
\end{figure}

Next, suppose $\alpha,\phi(\alpha)$ is a simple symmetric pair of connectors.
These cases are illustrated in Figure~\ref{symrank fig 2}.

\noindent
\textbf{Case 3.}
If $E$, $E'$, $\phi(E)$ and $\phi(E')$ are all distinct, then connecting $E$ to $E'$ by $\alpha$ and $\phi(E)$ to $\phi(E')$ by $\phi(\alpha)$ decreases $b_0^\phi$ by $1$ and leaves $b_1^\phi$ unchanged.

\noindent
\textbf{Case 4.}
If $E=E'$ but $E\neq\phi(E)=\phi(E')$, then $b_0^\phi$ is unchanged, while $b_1^\phi$ increases by~$1$.

\noindent
\textbf{Case 5.}
If $E=\phi(E)$ but $E$, $E'$ and $\phi(E')$ are all distinct, then $b_0^\phi$ decreases by $1$ and $b_1^\phi$ is unchanged.  

\noindent
\textbf{Case 6.}
If $E=\phi(E)$ but $E\neq E'=\phi(E')$, then $b_0^\phi$ is unchanged, while $b_1^\phi$ increases by $1$.
(A new non-bounding circle passes through $E$, $\alpha$, $E'$, and $\phi(\alpha)$, back to $E$, and $\phi$ reverses the orientation of this circle or sends it to a distinct circle.)

\noindent
\textbf{Case 7.}
If $E=\phi(E')$ but $E\neq\phi(E)=E'$, then $b_0^\phi$ decreases by $1$ and $b_1^\phi$ is unchanged.
There are two subcases here, and two different reasons that $b_1^\phi$ is unchanged.

\noindent
\textit{Case 7a.}
In this case, $\alpha$ and $\phi(\alpha)$, together with curves in $E$ and $E'$ bound a disk containing double points $d_+,d_-$ with $d_+$ in $E$ and $d_-$ in $E'$.
(This case is mentioned specifically in Definition~\ref{sym simp conn} and illustrated in Figure~\ref{diskfill fig}.)
Thus there are no new non-bounding circles created when $E$ and $E'$ are combined along $\alpha$ and $\phi(\alpha)$.

\noindent
\textit{Case 7b.}
In this case, $\alpha$ and $\phi(\alpha)$, together with curves in $E$ and $E'$ form a new non-bounding circle.
However, this circle can be chosen so that $\phi$ maps it to itself, with the same orientation, so it does not contribute to $b_1^\phi$.
The picture of this case in Figure~\ref{symrank fig 2} shows an annulus with a disk removed, no double points, and one $\phi$-fixed point.
The action of $\phi$ fixes the removed disk and rotates its boundary by a half-turn.

\noindent
\textbf{Case 8.}
If $E=E'=\phi(E)=\phi(E')$, then $b_0^\phi$ is unchanged, while $b_1^\phi$ increases by~$1$.
For reasons analogous to Case 7, there are two subcases and two reasons that $b_1^\phi$ increases by~$1$.

\noindent
\textit{Case 8a.}
In this case, $\alpha$ and $\phi(\alpha)$, together with two curves in $E$ bound a disk containing double points $d_+,d_-$ with $d_+$ and $d_-$ at two different locations on the boundary of $E$.
(This case is mentioned specifically in Definition~\ref{sym simp conn}.)
There is one new non-bounding circle created because $E$ is connected to itself along $\alpha$ and $\phi(\alpha)$, and this circle can be chosen so that $\phi$ reverses its orientation.

\noindent
\textit{Case 8b.}
In this case, $\alpha$ and $\phi(\alpha)$ combine with $E$ to increase the dimension of $H_1$ by two.  
Two new independent non-bounding circles can be constructed as a symmetric pair (one following $\alpha$ and one following $\phi(\alpha)$), which together contribute~$1$ to $b_1^\phi$.
The picture of this case in Figure~\ref{symrank fig 2} shows the same surface as Case~7b.

\begin{figure}
\begin{tabular}{ccc}
\scalebox{0.5}{\includegraphics{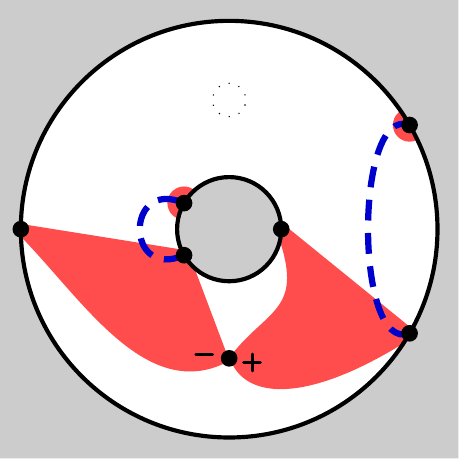}}&
\scalebox{0.5}{\includegraphics{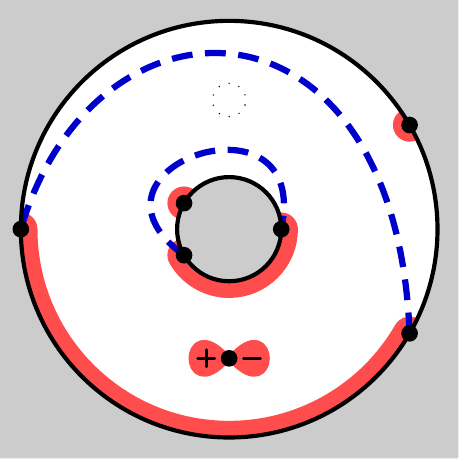}}&
\scalebox{0.5}{\includegraphics{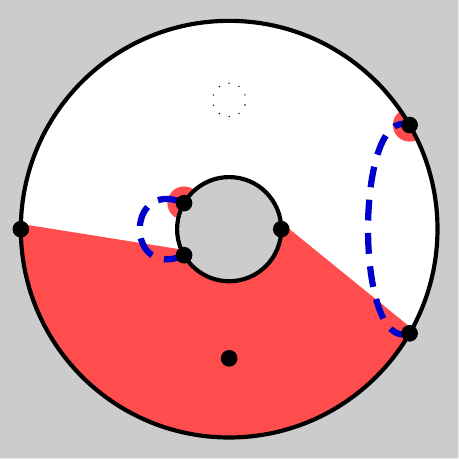}}\\
Case 3&Case4&Case 5\\[5pt]
\scalebox{0.5}{\includegraphics{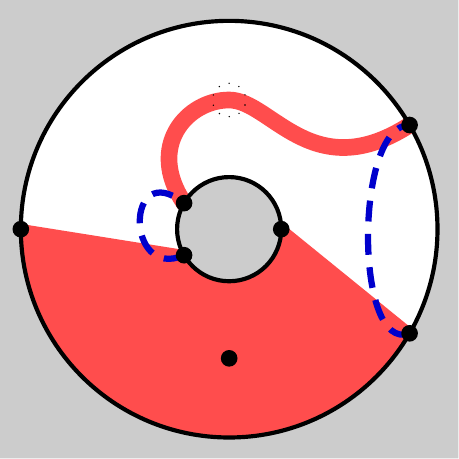}}&
\scalebox{0.5}{\includegraphics{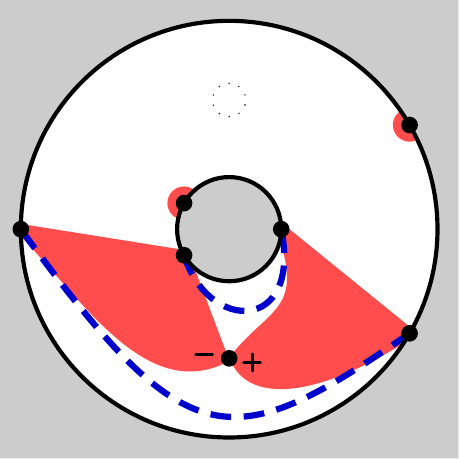}}&
\scalebox{0.5}{\includegraphics{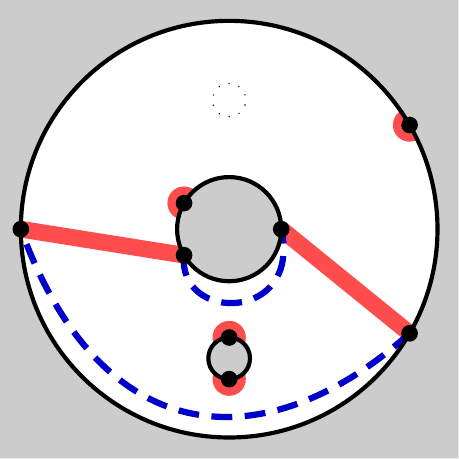}}\\
Case 6&Case 7a&Case 7b\\[5pt]
\end{tabular}
\begin{tabular}{cc}
\scalebox{0.5}{\includegraphics{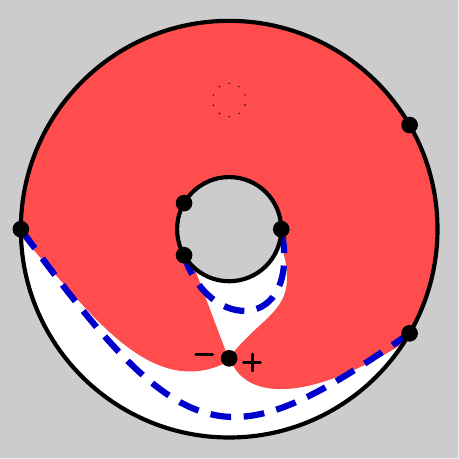}}&
\scalebox{0.5}{\includegraphics{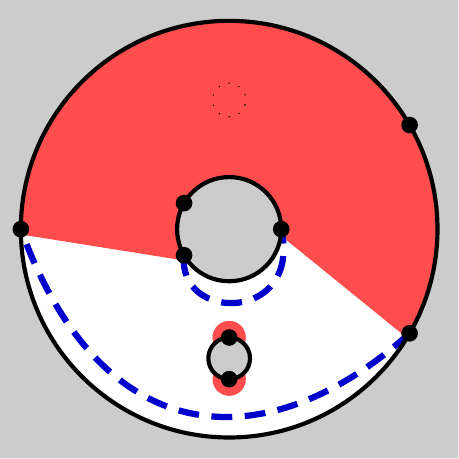}}\\
Case 8a&Case 8b\\
\end{tabular}
\caption{More illustrations for the proof of Theorem~\ref{sym graded}}
\label{symrank fig 2}
\end{figure}

The restrictions in Definition~\ref{sym simp conn} ensure that the result of joining $E$, $E'$, $\phi(E)$ and $\phi(E')$ along $\alpha$ and $\phi(\alpha)$ is a disjoint collection of subsets that satisfy the definition of embedded blocks, except that the subset containing $E$ and $E'$ may fail to be a block because the same ring may appear more than once as a component of its boundary.

The last step in the construction of $\P\cup\alpha\cup\phi(\alpha)$ is to combine blocks in the collection whose boundaries contain the same ring (either combining different blocks or combining blocks to themselves).
Suppose $F$ and $F'$ are blocks whose boundaries contain the same ring $U$.
The $\phi$-symmetry of the construction means that $\phi(F)$ and $\phi(F')$ both contain the ring $\phi(U)$.
(It is impossible to have $\phi(U)$ isotopic to $U$, because symmetric rings cannot form part of the boundary of an embedded block, and because a symmetric pair of rings may not bound an annulus in $\S$ containing no double points.)
We consider all possibilities of coincidences among the blocks $F$, $F'$, $\phi(F)$, and $\phi(F')$, showing in every case that $b_1^\phi(\P)-b_0^\phi(\P)$ is unchanged when all the blocks are joined.
(We check all cases up to swapping $F$ and $F'$ and/or swapping $F$ and $\phi(F)$.)
The cases are illustrated in Figure~\ref{symrank fig 3}.
The pictures in the figure don't represent a specific surface $\S$, but rather represent parts of the blocks $F$, $F'$, $\phi(F)$, and $\phi(F')$ near $U$ and $\phi(U)$.
Black boundary lines for the blocks are shown near $U$ and $\phi(U)$ to clearly show the added annuli, but away from the added annuli, boundary lines are omitted, to emphasize that the picture omits topological features, boundary segments, etc.
All coincidences between the blocks are shown by connections between the blocks in the picture.
Black boundary lines are also used in Case~d to show that the two blocks shown do not coincide.
In each case, the symmetry $\phi$ acts as a half-turn rotation about the center of the figure.

\begin{figure}
\begin{tabular}{ccc}
\scalebox{0.5}{\includegraphics{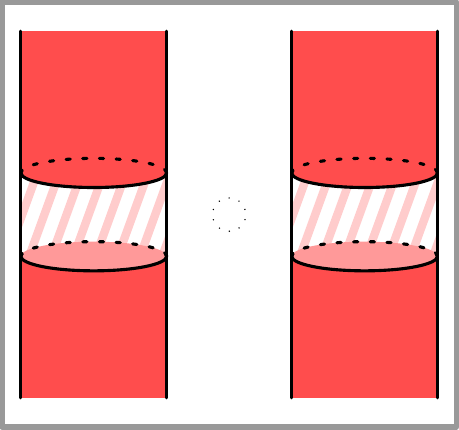}}&
\scalebox{0.5}{\includegraphics{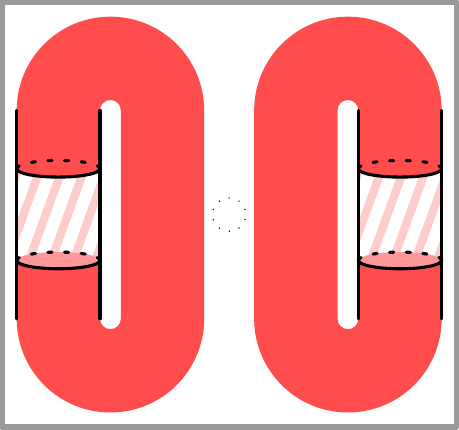}}&
\scalebox{0.5}{\includegraphics{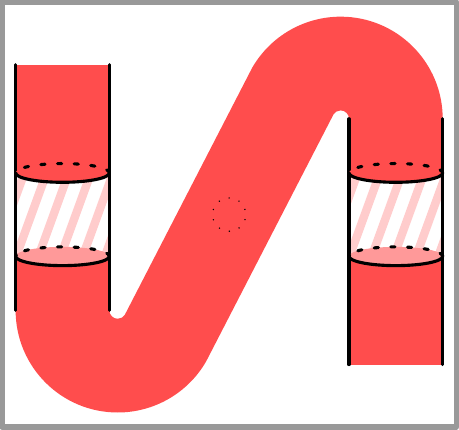}}
\\
Case a&Case b&Case c\\[5pt]
\scalebox{0.5}{\includegraphics{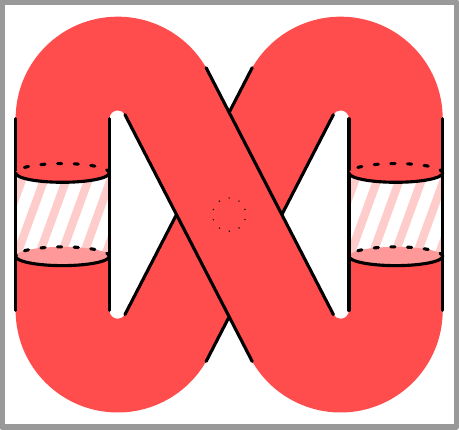}}&
\scalebox{0.5}{\includegraphics{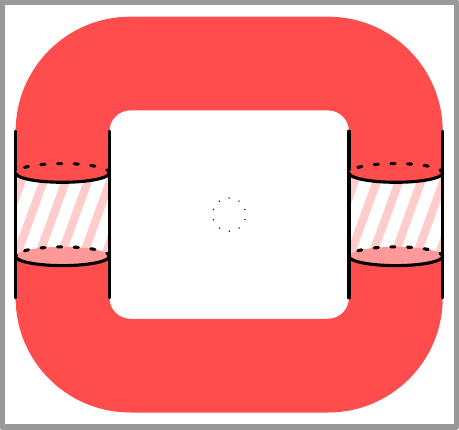}}&
\scalebox{0.5}{\includegraphics{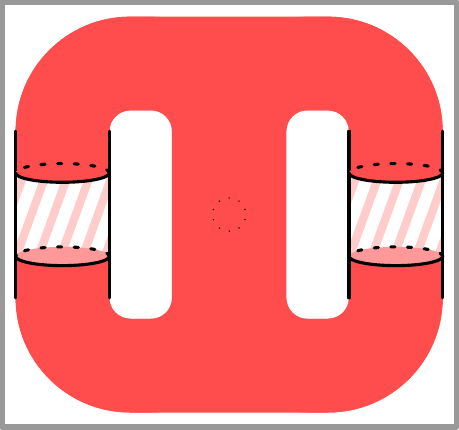}}
\\
Case d&Case e&Case f
\end{tabular}
\caption{Still more illustrations for the proof of Theorem~\ref{sym graded}}
\label{symrank fig 3}
\end{figure}

\noindent
\textbf{Case a.}
If $F$, $F'$, $\phi(F)$, and $\phi(F')$ are all distinct, then $b_0^\phi$ decreases by $1$ and $b_1^\phi$ decreases by $1$.
(The symmetric pair of rings $U,\phi(U)$ counts once after the joining rather than twice.)

\noindent
\textbf{Case b.}
If $F=F'$ but $F\neq\phi(F)=\phi(F')$, then $b_0^\phi$ and $b_1^\phi$ are unchanged.
(The symmetric pair $U,\phi(U)$ counts once rather than twice, but a new symmetric pair of non-bounding circles is created, one circle crossing each added annulus.) 

\noindent
\textbf{Case c.}
If $F=\phi(F)$ but $F$, $F'$, and $\phi(F')$ are all distinct, then $b_0^\phi$ decreases by $1$ and $b_1^\phi$ decreases by $1$.
(Again, $U,\phi(U)$ counts only once.)

\noindent
\textbf{Case d.}
If $F=\phi(F)$ but $F\neq F'=\phi(F')$, then $b_0^\phi$ and $b_1^\phi$ are unchanged. 
(Again, $U,\phi(U)$ counts only once.
A new non-bounding circle is formed, crossing both added annuli, and $\phi$ reverses its orientation or sends it to a different circle.)

\noindent
\textbf{Case e.}
If $F=\phi(F')$ but $F\neq\phi(F)=F'$, then $b_0^\phi$ and $b_1^\phi$ both decrease by $1$.
(Again, $U,\phi(U)$ counts only once.
A new non-bounding circle is formed, crossing both added annuli, but this circle can be chosen so that $\phi$ maps it to itself, with the same orientation.)  

\noindent
\textbf{Case f.}
If $F=F'=\phi(F)=\phi(F')$, then $b_0^\phi$ and $b_1^\phi$ are unchanged. 
(Again, $U,\phi(U)$ counts only once, but there is a new symmetric pair $C,\phi(C)$ of non-bounding circles, with $C$ crossing the annulus associated to $U$ and not crossing the annulus associated to $\phi(U)$.)
%
\end{proof}

\section{Noncrossing partitions of classical finite and affine types}\label{ex sec}
In this section, we revisit the examples that motivated the constructions in this paper, namely noncrossing partitions of finite types A, B, and D and affine types $\afftype{A}$, $\afftype{B}$, $\afftype{C}$, and $\afftype{D}$.
In a general Coxeter group, a \newword{poset of noncrossing partitions} is the interval $[1,c]_T$, where $1$ is the identity element, $c$ is a special kind of element called a \newword{Coxeter element}, and the subscript $_T$ refers to a particular partial order $\le_T$ on the elements of the group.
Details on the definition are found in many references, including in \cite{affncA} which begins the treatment of planar models for classical affine types.

\begin{remark}\label{clus rem}
We will see, in these examples, some parallels between noncrossing partitions of marked surfaces and the \emph{triangulations} of marked surfaces that model the combinatorics of certain cluster algebras \cite{cats1,cats2} as well as the related triangulations of marked orbifolds that appear in \cite{FeShTu12a,FeShTu12b,FeTu}.
As is well known, noncrossing partitions and cluster algebras of finite type A are modeled on a disk, while the type-B models live on a disk with a central symmetry or equivalently on the quotient disk (a disk with one degree-$2$ orbifold point).
The analogy continues, so that in type $\afftype{A}$, both cluster algebras and noncrossing partitions are modeled on an annulus.
Furthermore, the type $\afftype{C}$ noncrossing partition models live on a symmetric annulus or its quotient, a disk with two degree-$2$ orbifold points, while the cluster algebra models also live on the quotient.

The planar model for cluster combinatorics in type D$_n$ features a disk with $n$ boundary marked points and one puncture.
The model for noncrossing partitions features a disk with a central symmetry, $2n$ boundary marked points, and a pair of double points.
Similarly, type-$\afftype{D}$ cluster combinatorics takes place on a disk with two punctures, while type-$\afftype{D}$ noncrossing partitions naturally live on an annulus with an involutive symmetry and two pairs of double points.
In the finite or affine type-D or type-$\afftype{D}$ noncrossing partition models, it would seem natural to pass to the quotient modulo the symmetry and at the same time turn the double points into punctures, thus obtaining a disk with one or two punctures.
(Similar considerations apply in type~$\afftype{B}$.)

However, attempts to construct models of noncrossing partitions in these punctured settings have failed to yield natural, simple models, especially in affine type~$\afftype{D}$.
We have come to the conclusion that passing to the quotient and turning double points into punctures is not the natural setting for noncrossing partitions.
(Apparently, just the opposite is true for cluster combinatorics:
The ``tagged triangulations'' of \cite{cats1,cats2} constitute a natural, simple model incorporating punctures.
In finite type D, playing history backwards, one can pass from tagged triangulations of a disk with one puncture to a centrally symmetric model with what we might call ``doubled diameters'' \cite{ga}.
But it seems difficult to replace triangulations of general punctured marked surfaces with any natural, simple model that imposes an involutive symmetry and ``doubles'' something, in exchange for abolishing punctures.)
\end{remark}

\subsection{Type A}\label{A sec}
Let $(\S,\M)$ be the marked surface consisting of a disk $\S$ and a set~$\M$ of $n+1$ points on the boundary of $\S$.
Theorem~\ref{lattice} recovers, in this case, the main result of \cite{Kreweras}, namely that the ``noncrossing partitions of a cycle'' form a lattice.
Every embedded block in $\S$ is a point, curve, or disk, so Theorem~\ref{graded} recovers the known result that the rank function is $n+1$ minus the number of blocks.
Another key result of \cite{Kreweras} is an explicit anti-automorphism of the noncrossing partition lattice, but we have not found such a ``Kreweras complementation'' map for general marked surfaces.  
Biane gave an explicit bijection \cite[Theorem~1]{Biane} between the $\NC_{(\S,\M)}$ and the lattice $[1,c]_T$ in the Coxeter group $S_{n+1}$ (of type A$_n$).

\subsection{Type B}\label{B sec}
Let $\S$ be a disk, let $\phi$ be a two-fold rotational symmetry of $\S$, and let $\B$ be a set of $2n$ boundary points that is permuted by~$\phi$.
We take $\D=\emptyset$, so that $\S=\S^\pm$.
Theorem~\ref{sym graded} recovers part of \cite[Proposition~2]{Reiner}, namely that $\NC_{(\S,\B,\emptyset,\phi)}$ is graded, with rank function given by $n$ minus the number of symmetric pairs of distinct embedded blocks.
(Again, all embedded blocks are points, curves, or disks.)
The remaining assertion of \cite[Proposition~2]{Reiner} is that $\NC_{(\S,\B,\emptyset,\phi)}$ is a lattice.
In fact, the map $\phi$ induces an automorphism of $\NC(\S,\B)$, and $\NC_{(\S,\B,\emptyset,\phi)}$ is the subposet of $\NC(\S,\B)$ induced by noncrossing partitions fixed (up to isotopy) by $\phi$.
It is easy and well known that, given a lattice $L$ and an automorphism $\eta$ of~$L$, the set of fixed points of $\eta$ constitutes a sublattice of $L$, and we recover the fact that $\NC_{(\S,\B,\emptyset,\phi)}$ is a lattice.

General ``folding'' considerations (discussed in \cite[Section~2.3]{affncA}) show that the interval $[1,c]_T$ in a Coxeter group of finite type B$_n$ is the sublattice of the analogous interval in $S_{2n}$ consisting of elements of $S_{2n}$ fixed by a certain involutive automorphism.
The corresponding automorphism of $\NC(\S,\B)$ is induced by $\phi$, so the interval $[1,c]_T$ in type B$_n$ is isomorphic to $\NC_{(\S,\B,\emptyset,\phi)}$, as pointed out in~\cite{Bessis,BGN,Bra-Wa}.

Interestingly, the type-B noncrossing partition lattice can also be constructed as a lattice of noncrossing partitions of an ordinary marked surface.
The following is an observation due to Laura Brestensky in connection with~\cite{BThesis}.
It was also noticed by Alexandersson, Linusson, Potka, and Uhlin \cite[Section~9]{ALPU} or can be obtained by combining \cite[Proposition~7.6]{McSul} with results of \cite[Section~5.2]{affncA}.

\begin{theorem}\label{B orb empty}
The type-B$_n$ noncrossing partition lattice is isomorphic to the noncrossing partition lattice $\NC_{(\A,\M)}$, where $\A$ is an annulus and $\M$ consists of $n$ marked points on one boundary of $\A$ and no marked points on the other boundary.
\end{theorem}

The proof of the theorem is simple, and we sketch it here.
The type-B noncrossing partition lattice is isomorphic to $\NC_{(\S,\B,\emptyset,\phi)}$.
The quotient $\S/\phi$ is another disk with $n$ marked points on its boundary and a degree-$2$ orbifold point in its interior.
Symmetric embedded blocks/pairs of embedded blocks in $\S$ become points, curves between marked points, curves between a marked point and the orbifold point, or disks in $\S/\phi$, and the disk blocks may or may not contain the orbifold point.
Replacing the orbifold point with an empty boundary component, we obtain the annulus $\A$ with marked points $\M$ as in the theorem.
Replacing each disk block containing the orbifold point (and each curve connecting a marked point to the orbifold point) with an annular block having the empty boundary component as one of its boundaries, we obtain an isomorphism from $\NC_{(\S,\B,\emptyset,\phi)}$ to $\NC_{(\A,\M)}$.

It is tempting to generalize to a model where orbifold points ``are'' empty boundary components, but this idea does not hold up in general (for example in Section~\ref{C tilde sec}).

\subsection{Type D}\label{D sec}
Let $\S$ be a disk with a two-fold rotational symmetry $\phi$ and a set $\B$ of boundary points permuted by~$\phi$.
Take $\D=\set{d}$, for $d$ the fixed point of $\phi$, so that $\D^\pm=\set{d_+,d_-}$.
Reuse the symbol $\phi$ for the map on $\S^\pm$ that agrees with $\phi$ on $\S\setminus\D$ and sends $d_+$ to $d_-$ and vice versa.
Then $\NC_{(\S^\pm,\B,\D^\pm,\phi)}$ is the type-D noncrossing partition lattice defined in~\cite{Ath-Rei}, which is isomorphic to the interval $[1,c]_T$ in type D$_n$, by \cite[Theorem~1.1]{Ath-Rei}.
Theorem~\ref{sym graded} recovers part of \cite[Proposition~3.1]{Ath-Rei}, which says that $\NC_{(\S^\pm,\B,\D^\pm,\phi)}$ is a graded lattice whose rank function is $n$ minus the number of symmetric pairs of distinct embedded blocks.

\subsection{Type $\afftype{A}$}\label{A tilde sec}
Let $\S$ be an annulus and let $\M$ be a set of $n$ marked points on the boundary of $\S$, with at least one marked point on each component of the boundary.
Embedded blocks in $(\S,\M)$ are points, curves, disks, or annuli.
Theorems~\ref{lattice} and~\ref{graded} specialize to the following theorem, which is a result of \cite{BThesis}.

\begin{theorem}\label{A tilde main}
The poset $\NC_{(\S,\M)}$ of noncrossing partitions of an annulus with marked points on both boundaries, $n$ marked points in all, is a graded lattice, with rank function given by $n$ minus the number of blocks that are not annuli.
\end{theorem}

Coxeter groups of type $\afftype{A}$ have the property (unique among affine or finite Coxeter groups) that two different Coxeter elements can fail to be conjugate.
For non-conjugate Coxeter elements $c$ and $c'$, the intervals $[1,c]_T$ and $[1,c']_T$ can fail to be isomorphic.
The choice of Coxeter element can be encoded as a choice, for each number $i$ from $1$ to $n$, of whether to put the $i^{\text{th}}$ point on the inner or outer boundary of the annulus.
The results quoted below assume this correspondence between Coxeter elements and placements of marked points.

We say that an annular block in $(\S,\M)$ is \newword{dangling} if it contains marked points from only one boundary of $\S$.
Equivalently, one component of its boundary is the unique ring in $\S$.
Among the main results of \cite{BThesis} (which appears as part of \mbox{\cite[Theorem~3.25]{affncA}}) is that the interval $[1,c]_T$ in the Coxeter group of type $\afftype{A}_{n-1}$ is isomorphic to the subposet of $\NC_{(\S,\M)}$ consisting of noncrossing partitions with no dangling annular blocks.
Another main result is the construction of a larger group such that the analog of $[1,c]_T$ is isomorphic to $\NC_{(\S,\M)}$.
These results are proved by means of an explicit isomorphism that reads the cycle structure of an affine permutation from a noncrossing partition.

The construction of planar diagrams for elements of $[1,c]_T$ follows a process analogous to the constructions in \cite{plane} in finite type, where the diagrams discussed in Sections~\ref{A sec}--\ref{D sec} were obtained by projecting a small orbit to the Coxeter plane. 
The idea of extending $[1,c]_T$ to a larger poset that is a lattice is inspired by the main construction of \cite{McSul}, which proceeds by ``factoring'' translations.
We show in \cite{affncA} that extending the poset of noncrossing partitions with no dangling annular blocks to the full lattice $\NC_{(\S,\M)}$ also proceeds by factoring translations, but in a different way suggested by the combinatorics of noncrossing partitions of $(\S,\M)$.

\subsection{Type $\afftype{C}$}\label{C tilde sec}
Let $\S$ be an annulus and let $\B$ be a set of $2n-2$ marked points, with $n-1$ marked points on each boundary component of $\S$.
Let $\phi$ be a nontrivial involutive homeomorphism from $\S$ to itself that exchanges the two boundary components of~$\S$ and permutes $\B$.
Picturing $\S$ as a hollow cylindrical tube, we can see $\phi$ as a rotation by a half-turn through an axis parallel to the boundary circles of the tube.
In particular, $\phi$ has two fixed points in the interior of $\S$.
We take $\D=\emptyset$, so that $\S=\S^\pm$.
The embedded blocks of $(\S,\B,\emptyset,\phi)$ are points, curves, disks and annuli, but a dangling annulus cannot appear as a block in a noncrossing partition of $(\S,\B,\emptyset,\phi)$, because such a block $E$ would contain the same ring in its boundary as its symmetric partner $\phi(E)$.
The number of annular blocks in a noncrossing partition is either $0$ or $1$.
The poset $\NC_{(\S,\B,\emptyset,\phi)}$ is the subposet of $\NC(\S,\B)$ induced by noncrossing partitions fixed (up to isotopy) by $\phi$, and thus in particular is a lattice.
Combining this fact with Theorem~\ref{sym graded}, we obtain the following theorem, which again is a result of \cite{BThesis}.

\begin{theorem}\label{C tilde main}
The poset $\NC_{(\S,\B,\emptyset,\phi)}$ of symmetric noncrossing partitions of an annulus with $n-1$ marked points on each boundary is a graded lattice, with rank function given by $n-1$ minus the number of symmetric pairs of distinct disk blocks plus the number of annular blocks.
\end{theorem}

The same folding considerations mentioned in Section~\ref{B sec}, and the fact that symmetric noncrossing partitions can't have dangling annular blocks, show that the interval $[1,c]_T$ in the Coxeter group of type $\afftype{C}_{n-1}$ is isomorphic to the sublattice of the analogous interval in $\afftype{A}_{2n-2}$ consisting of elements of fixed by a certain involutive automorphism.
In light of the folding results mentioned in Section~\ref{A tilde sec}, the interval $[1,c]_T$ in type $\afftype{C}_{n-1}$ is isomorphic to $\NC_{(\S,\B,\emptyset,\phi)}$, as pointed out in \cite{BThesis}.
The lattice $\NC_{(\S,\B,\emptyset,\phi)}$ can also be realized as noncrossing partitions of a disk with two orbifold points.
Details are found in \cite[Section~4]{affncA}.

\subsection{Type $\afftype{D}$}\label{D tilde sec}
For $n\ge 3$, let $\S$ be an annulus and let $\B$ be a set of $2n-6$ marked points on the boundary of $\S$, with $n-3$ marked points on each component of the boundary.
Let $\phi$ be a nontrivial involutive homeomorphism from $\S$ to itself, as described in Section~\ref{C tilde sec}, permuting $\B$.
Let $\D$ be the set consisting of the two fixed points of $\phi$, so that $\D^\pm$ consists of four points, two at each fixed point.
Symmetric embedded blocks in $(\S^\pm,\B,\D^\pm,\phi)$ can be disks or annuli, while symmetric pairs of blocks can be pairs of points, curves, disks, or annuli.
Theorem~\ref{sym graded} specializes to the following theorem, a result of \cite{BThesis}.

\begin{theorem}\label{D tilde main}
The poset $\NC_{(\S^\pm,\B,\D^\pm,\phi)}$ of symmetric noncrossing partitions of an annulus with $n-3$ marked points on each boundary and two pairs of double points is graded, with rank function given by $n-1$ minus the number of symmetric pairs of distinct non-annular blocks plus the number of symmetric annular blocks.
\end{theorem}

The poset $\NC_{(\S^\pm,\B,\D^\pm,\phi)}$ is not a lattice, because it has an interval isomorphic to the noncrossing partition lattice described in Example~\ref{non lat} and shown in Figure~\ref{non lat fig}.
This is the interval below the noncrossing partition whose only nontrivial block is a symmetric annulus with both double points in its interior and no marked points on its boundary.
(An erroneous proof that $\NC_{(\S^\pm,\B,\D^\pm,\phi)}$ is a lattice in this case was given in~\cite{BThesis}.
Essentially the same error was made in an earlier version of this paper, which purported to prove that $\NC_{(\S^\pm,\B,\D^\pm,\phi)}$ is a lattice in full generality.)

In this type-$\afftype{D}$ context, an annular block is \newword{dangling} if it contains marked points only from one boundary of $\S$ or if its only marked points are double points.
Equivalently, one or more components of its boundary is a ring.
One of the main results of \cite{BThesis}, reported in \cite[Section~5.2]{affncD}, is that the interval $[1,c]_T$ in the Coxeter group of type~$\afftype{D}_{n-1}$ is isomorphic to the subposet of $\NC_{(\S^\pm,\B,\D^\pm,\phi)}$ consisting of noncrossing partitions with no dangling annular blocks.
One can again extend the Coxeter group of type $\afftype{D}_{n-1}$ to a larger group such that $\NC_{(\S^\pm,\B,\D^\pm,\phi)}$ is isomorphic to the interval analogous to $[1,c]_T$ in the larger group.
The extension to a larger group again proceeds by factoring translations (see Section~\ref{A tilde sec}), but to obtain a lattice, we need further factorization of translations which move the model beyond the noncrossing partition posets defined in this paper.
See \cite[Section~6]{affncD}.

\subsection{Type $\afftype{B}$}\label{B tilde sec}
The last example in this section is type $\afftype{B}_{n-1}$, which is obtained from type $\afftype{D}_n$ by folding.
Folding, together with the results mentioned in Section~\ref{D tilde sec}, suggests that we should pass from the lattice of symmetric noncrossing partitions of an annulus with two pairs of double points to a particular sublattice.
If $d_+,d_-$ is one of the pairs of double points, then the sublattice consists of noncrossing partitions fixed under the symmetry of swapping $d_+$ with $d_-$.
These are the partitions which either have blocks $\set{d_+}$ and $\set{d_-}$ or have a block whose interior contains $\set{d_+,d_-}$.
In effect, we might as well delete $d_+$ and $d_-$ entirely.

Accordingly, let $\S$ be an annulus and let $\B$ be a set of $2n-4$ marked points on the boundary of $\S$, with $n-2$ marked points on each component of the boundary.
Take $\phi$ to be the same symmetry as before, permuting $\B$, and let $\D$ be the set consisting of one of the fixed points of $\phi$, so that $\D^\pm$ consists of one pair of double points at that fixed point.
In this case, Theorem~\ref{sym graded} specializes to the following theorem.

\begin{theorem}\label{B tilde main}
The poset $\NC_{(\S^\pm,\B,\D^\pm,\phi)}$ of symmetric noncrossing partitions of an annulus with $n-2$ marked points on each boundary and one pair of double points is graded, with rank function given by $n-1$ minus the number of symmetric pairs of distinct non-annular blocks plus the number of symmetric annular blocks.
\end{theorem}

In \cite[Section~7]{affncD}, we show that the interval $[1,c]_T$ in the Coxeter group of type $\afftype{B}_{n-1}$ is isomorphic to the subposet of $\NC_{(\S^\pm,\B,\D^\pm,\phi)}$ consisting of noncrossing partitions with no dangling annular blocks, and realize $\NC_{(\S^\pm,\B,\D^\pm,\phi)}$ as an interval in a supergroup.
We also show that $\NC_{(\S^\pm,\B,\D^\pm,\phi)}$ is isomorphic to the lattice constructed in \cite{McSul}.

\end{document}